\documentclass[11pt]{article}
\usepackage{amsfonts,amsmath,epsfig}
\usepackage{blkarray}
\usepackage{cleveref}
\usepackage{arydshln}  
\usepackage{authblk}

\newcommand{\be}{\mathbf e}

\newcommand{\bw}{\mathbf w}
\newcommand{\bx}{\mathbf x}
\newcommand{\by}{\mathbf y}
\newcommand{\bz}{\mathbf z}
\newcommand{\bv}{\mathbf v}
\newcommand{\bu}{\mathbf u}
\newcommand{\calh}{{\mathcal H}}
\newcommand{\calm}{{\mathcal M}}
\newcommand{\calo}{{\mathcal O}}

\newcommand{\cals}{{\mathcal S}}
\newcommand{\blambda}{{\boldsymbol \lambda}}
\newcommand{\bomega}{{\boldsymbol \omega}}
\newcommand{\bxi}{{\boldsymbol \xi}}
\newcommand{\zero}{\mathbf 0}
\newcommand{\ph}{\phantom}
\newcommand{\fns}{\footnotesize}
\newcommand{\wh}{\widehat}
\newcommand{\wt}{\widetilde}
\newcommand{\sca}[1]{\scalebox{.8}{$#1$}}

\newcommand{\mtxb}[1]{\left[ \begin{matrix} #1 \end{matrix} \right]}

\newcommand{\smtxa}[2]{
{\mbox{\footnotesize $\left[\!\! \begin{array}{#1} #2 \end{array} \!\! \right]$}}}

\newcommand{\smath}[1]{{\mbox{\scriptsize $#1$}}}

\usepackage{xcolor}

\newcommand{\ch}[1]{{\color{blue}#1}}  
\newcommand{\chh}[1]{{\color{orange}#1}}  
\renewcommand{\ch}[1]{#1} 
\renewcommand{\chh}[1]{#1} 

\newtheorem{theorem}{Theorem}

\newcounter{conjecture}
\newtheorem{example}[conjecture]{Example}
\newcounter{lemmacnt}
\newtheorem{lemma}[lemmacnt]{Lemma}

\newenvironment{proof}[1][Proof]{\textbf{#1. }}{\ \rule{0.5em}{0.5em}}

\newcommand{\R}{\mathbb{R}}
\newcommand{\C}{\mathbb{C}}
\newcommand{\PP}{\mathbb{P}}

\makeatletter \let\cl@part\relax \makeatother

\begin{document}

\newcommand{\rank}{\mathop{\rm rank}\nolimits}
\title{Numerical methods for rectangular multiparameter eigenvalue problems, with applications
\ch{to finding optimal} ARMA and LTI models}

\author[1]{Michiel E.~Hochstenbach}
\author[2]{Toma\v{z} Ko\v{s}ir}
\author[2]{Bor Plestenjak}

\affil[1]{\footnotesize Department of Mathematics and Computer Science, TU Eindhoven, PO Box 513, 5600 MB, The Netherlands}
\affil[2]{Faculty of Mathematics and Physics, University of Ljubljana, Jadranska 19, SI-1000 Ljubljana, Slovenia}
\renewcommand\Authands{ and }

\maketitle

\begin{abstract}
Standard multiparameter eigenvalue problems (MEPs)
are systems of $k\ge 2$ linear $k$-parameter square matrix pencils.
Recently, a new form \ch{of multiparameter eigenvalue problems} has emerged: a rectangular MEP (RMEP) with only one multivariate
rectangular matrix pencil, where we are looking for combinations of the parameters \ch{for which} the rank
of the pencil is not full.
Applications include finding the optimal least squares autoregressive moving average (ARMA) model
and the optimal least squares realization of autonomous linear time-invariant (LTI) dynamical system.
For linear and polynomial RMEPs, we give the number of solutions and show how these problems
can be solved numerically by a transformation into a
standard MEP. \ch{For the transformation we provide new linearizations for quadratic multivariate matrix polynomials with
a specific structure of monomials and consider mixed systems of rectangular and square multivariate matrix polynomials}.
\ch{This} numerical approach seems computationally considerably more attractive than the block Macaulay method, the only other currently available numerical method for polynomial RMEPs.\\

\noindent {\bf Keywords:} {Rectangular multiparameter eigenvalue problem,
\ch{standard} multiparameter eigenvalue problem, ARMA, LTI, block Macaulay matrix}\\
\end{abstract}
\section{Introduction}
The \ch{(standard)} \emph{multiparameter eigenvalue problem} (MEP) \cite{Atkinson}
has the form
\begin{equation}
 W_i(\blambda) \, \bx_i:=(V_{i0}+\lambda_1 V_{i1}+\cdots+\lambda_k V_{ik}) \, \bx_i = \zero,\quad i=1,\dots,k,
 \label{problem}
\end{equation}
where $V_{ij}\in\C^{n_i\times n_i}$ are \ch{square} matrices,
$\blambda=(\lambda_1,\dots,\lambda_k) \in \C^k$, and $\bx_i\in \C^{n_i}$ is nonzero.
In the generic case, problem \eqref{problem} has \chh{$n_1\cdots n_k$} \emph{eigenvalues} $\blambda\in \C^k$ that are roots of the
system of $k$ multivariate characteristic polynomials \ch{equations}
$\det(W_i(\blambda)) = 0$ for $i = 1, \dots, k$.
If $\blambda$ and nonzero vectors $\bx_1,\dots,\bx_k$ satisfy
\eqref{problem}, then the
tensor product $\bx = \bx_1 \otimes \cdots \otimes \bx_k$ is \ch{called} the corresponding \emph{eigenvector}.

A MEP \eqref{problem} is related to a system of generalized eigenvalue problems (GEPs)
\begin{equation} \label{drugi}
 \Delta_i \bz = \lambda_i \, \Delta_0 \bz, \qquad i = 1, \dots, k,
\end{equation}
where the \chh{$(n_1\cdots n_k) \times (n_1\cdots n_k)$} matrices
\[
\Delta_0 = \left|\begin{matrix}V_{11} & \cdots & V_{1k} \\[-0.5mm]
\vdots & & \vdots \\[-0.5mm]
V_{k1} & \cdots & V_{kk}\end{matrix}\right|_\otimes\ch{=\sum_{\sigma\in S_k}{\rm sgn}(\sigma) \, V_{1\sigma_1}\otimes V_{2\sigma_2}\otimes \cdots \otimes V_{k\sigma_k},}\]
\ch{where $\otimes$ denotes the Kronecker product}, and
\[
\Delta_i = -\ \left|\begin{matrix}V_{11} & \cdots & V_{1,i-1} & V_{10} & V_{1,i+1} & \cdots & V_{1k} \\[-0.5mm]
\vdots & & \vdots & \vdots & \vdots & & \vdots \\[-0.5mm]
V_{k1} & \cdots & V_{k,i-1} & V_{k0} & V_{k,i+1} & \cdots & V_{kk}\end{matrix}\right|_\otimes,\quad i=1,\ldots,k,
\]
are called \emph{operator determinants},
and $\bz = \bx_1 \otimes \cdots \otimes \bx_k$;
for \ch{more} details, see \cite{Atkinson}.
A generic MEP
\eqref{problem} is \emph{nonsingular}, which means that the corresponding
operator determinant $\Delta_0$ is nonsingular. In this case (see \cite{Atkinson}), the matrices $\Delta_0^{-1}\Delta_i$ for $i = 1,\dots,k$ commute and
the eigenvalues of \eqref{problem} and \eqref{drugi} agree.
By using this relation, a nonsingular MEP can
be numerically solved with the standard tools for GEPs; see, e.g., \cite{HKP}.
Several applications lead to singular MEPs, where $\Delta_0$ as well as all GEPs in \eqref{drugi} are singular; see, e.g, \cite{MP, KP}.
\ch{For singular MEPs}, we \ch{can} get the eigenvalues of \eqref{problem} from the common regular part of \eqref{drugi} \cite{MP}.

A generalization of \eqref{problem} are polynomial MEPs, where
$W_1,\dots,W_k$ are multivariate matrix polynomials. For instance, a \emph{quadratic two-parameter eigenvalue
problem} (quadratic 2EP) has the form
\begin{equation}
 (V_{i00}+\lambda V_{i10} + \mu V_{i01} + \lambda^2 V_{i20}+\lambda\mu V_{i11} + \mu^2 V_{i02}) \, \bx_i = \zero,\quad i=1,2,
 \label{q2ep}
\end{equation}
where $V_{ipq}\in \C^{n_i\times n_i}$.
A generic \ch{quadratic two-parameter eigenvalue} problem \eqref{q2ep} has $4n_1n_2$ eigenvalues $(\lambda, \mu)$ \cite{MP2}.

Recently, a new type of eigenvalue problems with $k\ge 2$ parameters has appeared. A generic form is
\begin{equation} \label{rectmep}
M({\blambda}) \, \bx := \Big(\sum_{\bomega} {\blambda}^{\bomega} A_{\bomega} \Big) \, \bx = \zero,
\end{equation}
where $\bomega = (\omega_1,\dots,\omega_k)$ is a multi-index, $\blambda^\bomega = \lambda_1^{\omega_1}\cdots \lambda_k^{\omega_k}$, and $\bx \in \C^n$ is nonzero;
again a $k$-tuple $\blambda=(\lambda_1,\dots,\lambda_k) \in \C^k$ is called an eigenvalue and $\bx$ is the corresponding eigenvector.
The key properties of this problem are:
\vspace{-2mm}
\begin{enumerate} \itemsep=-1mm
\item[a)] there is just \emph{one equation};
\item[b)] $A_\bomega=A_{\omega_1,\dots,\omega_k}$ are $(n+k-1) \times n$ \emph{rectangular} matrices over $\C$.
\end{enumerate}
We assume that the matrices are such that the normal rank of $M$, defined as
\[
{\rm nrank}(M):= \max_{\blambda\in \C^k} \, \rank(M(\blambda)),
\]
is full, i.e., ${\rm nrank}(M) = n$. Then $\blambda\in \C^k$ is an eigenvalue of problem \eqref{rectmep}
if $\rank(M(\blambda))<n$.
\ch{If the normal rank of $M$ is not full, we can still define
eigenvalues as points where $\rank(M(\blambda))<{\rm nrank}(M)$ (see, e.g., \cite{Khazanov}) but
such problems, which are more difficult to analyze and solve numerically, are
outside the scope of this paper.}

To distinguish it from \eqref{problem} and
\eqref{q2ep}, we call \eqref{rectmep} a \emph{rectangular multiparameter eigenvalue problem} (RMEP). The degree $d$ of \eqref{rectmep} corresponds
to the highest total degree \ch{$|\bomega|=\omega_1+\cdots+\omega_k$} of all its monomials $\blambda^\bomega$.
We assume that problem \eqref{rectmep} is such that is has finitely many solutions.
It does not seem easy to give necessary or sufficient conditions for this property;
heuristically, this property is expected to hold when the matrices $A_\bomega$ are ``generic''.
We give some partial results in \Cref{sec:linrmep} for the linear case $d = 1$, and in \Cref{sec:polyrmep} we show that
a generic polynomial RMEP of degree $d$ has $d^k\binom{n+k-1}{k}$ eigenvalues.

\ch{The first detailed study of RMEPs in a polynomial form \eqref{rectmep}
with more general $m\times n$ matrices, where $m\ge n$,
was done by Khazanov in \cite{Khazanov},
where the spectrum of such problems is defined and analyzed, but no numerical methods are considered.
Shapiro and Shapiro \cite{Shapiro} studied linear RMEPs},
where $M(\blambda)$ is a linear multivariate matrix pencil and the problem has the form
\begin{equation} \label{shap}
M(\blambda) \, \bx:=(A + \lambda_1 B_1 + \cdots +\lambda_k B_k) \, \bx = \zero,
\end{equation}
where $A,B_1,\dots,B_k\in \C^{(n+k-1) \times n}$.
It is shown in \cite{Shapiro} that a generic problem \eqref{shap} has $\binom{n+k-1}{k}$ eigenvalues, but again
no numerical methods are considered.
Before that, problem \eqref{shap} with $m\times n$ matrices, where $m\ge n$,
appeared under the name \ch{of} the \emph{eigentuple-eigenvector} problem
in \cite{BlumCurtis, BlumGeltner}, where
a gradient method is explored, and as a rank-reducing perturbation problem
in \cite{WDC}, where a Newton type approach is used 
to minimize the smallest singular value of $M(\blambda)$.
RMEPs in polynomial form \eqref{rectmep} have appeared in recent papers by De Moor \cite{DeMoor_LTI, DeMoor_CIS} and Vermeersch and De Moor \cite{DeMoor_ARMA, DeMoor_LAA}, \ch{see also \cite{DeMoor_SISC}.}
Applications include finding the optimal autoregressive moving average (ARMA) model and the optimal 
realization of autonomous linear time-invariant (LTI) dynamical system; \chh{these two identification problems will be the topic
of Sections~\ref{sec:ARMA} and \ref{sec:LTI}, respectively.}
The numerical methods for such problems in \cite{DeMoor_LTI, DeMoor_CIS, DeMoor_ARMA, DeMoor_LAA, DeMoor_SISC} use block Macaulay matrices which tend to be very large.
RMEPs of the form \eqref{shap} with applications in $\calh_2$-optimal model reduction have recently been studied by Alsubaie \cite{ALS},
who has derived a novel numerical method based on a reduction to a compressed MEP of the form \eqref{problem}.

We note that in the recent \cite{Tre21}, a combination of rectangular one-parameter matrix pencils and projections is considered; see also \cite{Tre02} for an earlier study.
These problems are not related as
problem \eqref{shap} for $k=1$ is equal to a GEP with square matrices and is different from the problem considered in \cite{Tre21}.

\ch{In this paper we show in \Cref{mepPB} that a generic linear RMEP \eqref{shap} can be transformed into a nonsingular MEP
of type \eqref{problem} by using $k$ deterministic or random projections (full-rank $n \times (n+k-1)$ operators in this context)
so that one can exploit available numerical methods for these MEPs. Based on this transformation we
provide an alternative derivation of
the compression technique by Alsubaie \cite{ALS}. With this technique we obtain from \eqref{shap} a system of GEPs of
substantially reduced size compared to the $\Delta$-matrices in \eqref{drugi}.
This approach yields (much) smaller matrices than the block Macaulay algorithm;
\ch{the} reduction factor is $n+k-1$.

For polynomial RMEPs \eqref{rectmep} we provide in \Cref{lem:fin_sol} the number of solutions, which generalizes the result
for the linear case from \cite{Shapiro}. We give two numerical approaches for polynomial RMEPs that we demonstrate on
the quadratic case. The first approach is based on a transformation
to a polynomial MEP by $k$ random projections, its advantage is that then one can use methods for polynomial MEPs that
work with the matrices $V_{ij}$. The second approach is to use generalized compression technique by Alsubaie that
leads to a system of singular GEPs. We apply this method, where in addition we provide new linearizations
for quadratic multivariate polynomial matrices
with a specific structure of monomials, to the RMEPs related to the computation of the optimal parameters of ARMA and LTI models.
We show that the new method is much more efficient than the block Macaulay method, which is the only other available numerical method for
RMEPs. The matrices are reduced by a factor ${\cal O}(N)$, where $N$ is the size of the sample.}

\ch{{\bf Outline.}} The rest of this paper has been organized as follows.
We show the basic idea of transforming an RMEP to a MEP in \Cref{sec:idea}.
In \Cref{sec:linrmep}, a theory that connects linear RMEPs to MEPs is provided first, after which a corresponding numerical method for linear RMEPs is proposed.
In \Cref{subs:compress} we present a compression technique from Alsubaie~\cite{ALS} that reduces the size of the problem, and
in \Cref{subs:Macaulay} we review the block Macaulay approach that is used to solve RMEPs in \cite{DeMoor_LTI, DeMoor_CIS, DeMoor_ARMA, DeMoor_LAA}.
In \Cref{sec:polyrmep} we discuss polynomial RMEPs and provide \ch{a proof on} the number of solutions in the generic case, while in \Cref{subs:num_polyrect}
we show numerical methods for such problems on the case of quadratic two-parameter RMEPs.
In \Cref{sec:ARMA} and \Cref{sec:LTI} we present in detail two particular cases of polynomial RMEPs related to ARMA and LTI models.
Relevant numerical examples are included in \ch{both} sections.
We end with conclusions in \Cref{sec:concl}.

\ch{{\bf Notation.}} Throughout the paper, $\|\cdot\|$ denotes the 2-norm. \ch{We denote tuples and vectors by boldface lowercase letters.}

\section{From rectangular to standard MEP} We show the basic idea in the following two examples.
\label{sec:idea}

\begin{example} \label[example]{ex:ex1}\rm We consider the problem
$(A+\lambda B+\mu C) \, \bx = \zero$ of type \eqref{shap}, where
\begin{equation}
A = \smtxa{cc}{1 & 2\\ 3 & 4 \\ 3 & 1}, \quad
B = \smtxa{cc}{1 & 3\\ 5 & 1 \\ 1 & 4}, \quad \textrm{and}\quad
C = \smtxa{cc}{4 & 1\\ 1 & 3 \\ 4 & 1}. \label{ex1:rect}
\end{equation}
From \cite[Lemma~1]{Shapiro}, we know that this problem has three eigenvalues.
To compute them, we multiply \eqref{ex1:rect} by
matrices $P_1$ and $P_2$ of size $2\times 3$,
which
results in a 2EP with matrices of size $2\times 2$ of the form
\begin{equation}
(P_iA+ \lambda \, P_iB + \mu \, P_iC) \, \bx_i = \zero, \quad i=1,2. \label{ex1:2ep}
\end{equation}
For generic $P_1$ and $P_2$ we get a nonsingular (see Section \ref{sec:linrmep} for the proof) 2EP \eqref{ex1:2ep} which has four eigenvalues $(\lambda,\mu)$ with
eigenvectors $\bx_1 \otimes \bx_2$.
Clearly, if $(\lambda,\mu)$ and $\bx \ne \zero$ solve \eqref{ex1:rect} then
$(\lambda,\mu)$ and $\bx \otimes \bx$ solve \eqref{ex1:2ep}. Since \eqref{ex1:2ep} has one more solution than \eqref{ex1:rect}, we get one
eigenvalue of \eqref{ex1:2ep}, where $\bx_1$ and $\bx_2$ are not colinear, that \chh{needs} to be ignored.

The elements of $P_1$ and $P_2$ may be drawn from a standard normal distribution; this technique is known as \emph{randomized sketching} and this idea is new in the MEP context. Alternatively, we can use deterministic selection matrices; the main goal \ch{in both cases} is to get a nonsingular 2EP if possible.
For \ch{this} particular problem,
the matrices $P_1$ and $P_2$ may be \emph{row selection matrices}, that is, their rows are canonical basis vectors.
As an example, we take $P_1= \smtxa{ccc}{1 & 0 & 0 \\ 0 & 1 & 0}$ and
$P_2= \smtxa{ccc}{0 & 1 & 0 \\ 0 & 0 & 1}$, i.e., we form \eqref{ex1:2ep} by omitting the last and the first row of \eqref{ex1:rect} respectively.
The 2EP \eqref{ex1:2ep} is
\begin{align}
 \left( \smtxa{ccc}{1 & 2\\ 3 & 4}
 + \lambda \, \smtxa{ccc}{1 & 3\\ 5 & 1}
 + \mu \, \smtxa{ccc}{4 & 1\\ 1 & 3}\right) \bx_1& = \zero,\nonumber\\[-2mm]
 \label{ex1:2epB} &\\[-2mm]
 \left( \smtxa{ccc}{3 & 4 \\ 3 & 1}
 + \lambda \, \smtxa{cc}{5 & 1 \\ 1 & 4}
 + \mu \, \smtxa{ccc}{1 & 3 \\ 4 & 1}\right) \bx_2& = \zero, \nonumber
\end{align}
with eigenvalues (to four decimal places)
$(2.6393, 3.0435)$,
$(-1.3577, 0.4365)$,
$(0.4553, -1.8007)$, and $(-0.3571, -1.2143)$,
where the first three eigenvalues are solutions of \eqref{ex1:rect}. In \Cref{subs:compress} we will show how
to reduce the size of the problem by using the compression technique from \cite{ALS} to get only solutions of the given linear RMEP.
\end{example}

\begin{example} \label{example2} \rm
The first-order ARMA(1,1) model (for more details see \ch{\cite{DeMoor_ARMA} and} \Cref{sec:ARMA}) leads to a quadratic rectangular 2EP (quadratic R2EP) of the form
\begin{equation} \label{ex2:armax}
(A+\lambda B + \mu C + \mu^2 D) \, \bx = \zero,
\end{equation}
where $A,B,C,D$ are $(3N-1) \times (3N-2)$ matrices and $N$ is the number of given data points.
Analogously to \Cref{ex:ex1}, we multiply the R2EP with random matrices $P_1$ and $P_2$ of
size $(3N-2) \times (3N-1)$ to obtain a quadratic 2EP
\[
 (P_iA+ \lambda \, P_iB + \mu \, P_iC + \mu^2 \, P_iD) \, \bx_i = \zero,\quad i=1,2.
\]
We can linearize the above problem as a singular 2EP; see \cite{HMP, MP}. For this particular example one option is to apply the linearization
\[
 \left(\mtxb{P_iA & P_iC \\ 0 & -I} +
 \lambda \, \mtxb{P_iB & 0 \\ 0 & 0}
 + \mu \, \mtxb{ 0 & P_iD \\ I & 0}\right)
 \mtxb{\bx_i \\ \mu \bx_i}= \zero, \quad i=1,2.
 \]
This singular 2EP may be solved with a staircase type algorithm from \cite{MP} applied to
$\Delta$-matrices of size $4(3N-2)^2\times 4(3N-2)^2$. We will show in \Cref{sec:ARMA} that \ch{this} size can be
further reduced by transforming the problem into a three-parameter MEP and by using the compression technique from Alsubaie \cite{ALS}.
The dimension \ch{still} quickly increases with $N$, but the  matrices are  considerably smaller than the block Macaulay matrices
used to solve such problems in \cite{DeMoor_LTI, DeMoor_CIS, DeMoor_ARMA, DeMoor_LAA}; \ch{see \Cref{tab:arma11}}.
\end{example}

\section{Linear rectangular MEPs} \label{sec:linrmep}
We will \ch{first} study the linear RMEP \eqref{shap}, a special case of problem \eqref{rectmep} with degree \ch{$d$ equal to} one.
We assume that the problem has full normal rank.
\ch{In \Cref{mepPB} we will show
that the approach that we demonstrated in \Cref{ex:ex1} can be applied to a generic
linear RMEP, and later we will provide a numerical method based on this theoretical results.}

The following result gives a sufficient condition on $B_1,\dots,B_k$
so that problem \eqref{shap} has finitely many eigenvalues for an arbitrary matrix $A$.
It is not hard to see that a necessary condition for the hypothesis in the following lemma is that all $B_i$ are of full rank $n$,
and that the $B_i$ are linearly independent in the space $\C^{(n+k-1) \times n}$.

\begin{lemma}[{\cite[Lemma~1]{Shapiro}}]\label[lemma]{lem:shapiro}
Let matrices $B_1,\dots,B_k\in \C^{(n+k-1) \times n}$ be such that for all $\blambda \ne \zero$,
${\rm rank}( \lambda_1 B_1+\cdots+\lambda_k B_k) = n$.
Then for each $A\in \C^{(n+k-1) \times n}$ the linear RMEP \eqref{shap} has exactly $\binom{n+k-1}{k}$ eigenvalues counting multiplicities.
\end{lemma}

A set $\cals \subseteq \C^r$ is called \emph{algebraic \ch{variety}} if it is the set of common zeros of finitely many
complex polynomials in $r$ variables, and \ch{$S$} is called \emph{proper} if $\cals \ne \C^r$. A set $\Omega\subset \C^r$
is said to be \emph{generic} if its complement is contained in a proper algebraic \ch{variety}.
\ch{We say that a property ${\cal P}$ holds generically if
there exists a generic set $\Omega$ such that
${\cal P}$ holds for all elements of $\Omega$.}
\ch{In this sense} we can show that the condition in \Cref{lem:shapiro} is satisfied for a generic set of matrices $B_1,\dots,B_k$.

\ch{\begin{theorem}\label{mepPB}
There exists a generic set $\Omega\subset (\C^{(n+k-1) \times n})^k$ with the following properties:
\begin{enumerate}
    \item[1)] For each $(B_1,\dots,B_k) \in \Omega$ there exists a generic set
        $\Theta\subset (\C^{n\times (n+k-1)})^k$ such that the operator determinant
        \begin{equation} \label{eq:delta0_prop4}
            \Delta_0 = \left|\begin{matrix} P_1B_1 & \cdots & P_1 B_k \\[-0.5mm]
            \vdots & & \vdots \\[-0.5mm]
            P_k B_1 & \cdots & P_k B_k\end{matrix}\right|_\otimes
        \end{equation}
        is nonsingular for all $(P_1,\dots,P_k) \in \Theta$;
    \item[2)] For each $(B_1,\dots,B_k) \in \Omega$ it holds that
        ${\rm rank}(\lambda_1 B_1+\cdots+\lambda_k B_k) = n$ for all $\blambda \ne \zero$.
\end{enumerate}
\end{theorem}

\begin{proof} We define
\begin{align*}
\Omega:=\big\{&(B_1,\dots,B_k)\in(\C^{(n+k-1) \times n})^k:\ \textrm{there exists }
(P_1,\dots,P_k)\in (\C^{n\times (n+k-1)})^k\ \textrm{such that}\\
& \eqref{eq:delta0_prop4}\ \textrm{is nonsingular}\big\}.
\end{align*}
Let us show that $\Omega$ is a generic set.
We fix the $n \times (n+k-1)$ matrices $P_1, \dots, P_k$ in \eqref{eq:delta0_prop4} to matrices
of the special form, where $P_j = [0_{n,j-1} \ \, I_n \ \, 0_{n,k-j}]$ is an $n \times n$ identity matrix
augmented by $j-1$ zero columns at the left and $k-j$ zero columns at the right.
Then $\det(\Delta_0)$ is a polynomial in the elements of $B_1,\dots,B_k$ and 
$\det(\Delta_0) = 0$ defines an algebraic \ch{variety}.
This set is proper since if we take $B_i = P_i^T$, $i = 1,\dots,k$, then it is easy to see that the corresponding
$\Delta_0$ is lower triangular with ones on the diagonal and $\det(\Delta_0) = 1$.
It follows that
\begin{align*}
\Omega_0:=\big\{&(B_1,\dots,B_k)\in(\C^{(n+k-1) \times n})^k:
\textrm{\eqref{eq:delta0_prop4} is nonsingular for}\
P_j = [0_{n,j-1} \ \, I_n \ \, 0_{n,k-j}],\\
& j=1,\ldots,k\big\}
\end{align*}
is a generic set and then $\Omega$ is generic due to $\Omega_0\subset \Omega$.

For 1), we take $(B_1,\dots,B_k) \in \Omega$, fix $B_1,\ldots,B_k$
in \eqref{eq:delta0_prop4} in a similar way as above, and consider $\det(\Delta_0)$ as a polynomial in the elements of $P_1,\dots,P_k$. It follows from the definition of $\Omega$ that there exist
matrices $P_1,\dots,P_k\in\C^{n\times (n+k-1)}$ such that $\det(\Delta_0)\ne 0$, therefore the
set of $\det(\Delta_0)=0$ defines a proper algebraic \ch{variety} and
there exists
a generic set $\Theta\subset (\C^{n\times (n+k-1)})^k$ that satisfies 1).

For 2), let $(B_1,\dots,B_k) \in \Omega$ and let
$\blambda\in \C^k$ and a nonzero $\bx$ be such that
$(\lambda_1B_1+\cdots+\lambda_kB_k) \, \bx= \zero$. We take
$(P_1,\dots,P_k)\in (\C^{n\times (n+k-1)})^k$ such that \eqref{eq:delta0_prop4} is nonsingular.
Then
\[
(\lambda_1 \, P_iB_1 + \cdots + \lambda_k \, P_iB_k) \, \bx_i = \zero,\quad i=1,\dots,k,
\]
is a nonsingular MEP of type \eqref{problem}.
From the associated system \eqref{drugi}, we get that $\lambda_i\,\Delta_0(\bx \otimes \cdots \otimes\bx) = \zero$ for $i = 1,\dots,k$.
Since $\Delta_0$ is nonsingular, it follows that $\blambda=\zero$, so indeed
 ${\rm rank}(\lambda_1 B_1+\cdots+\lambda_k B_k) = n$ for all $\blambda \ne \zero$.
\end{proof}
}

\ch{If follows from \Cref{mepPB} that for a generic choice of
matrices $B_1,\ldots,B_k$ the assumptions in \Cref{lem:shapiro} are satisfied and therefore}
a generic problem \eqref{shap} has \ch{exactly} $\binom{n+k-1}{k}$ solutions.
If $B_1,\dots,B_k$ and $P_1, \dots, P_k$ are such that the MEP
 \begin{equation} \label{eq:mepkalg}
 (P_iA+\lambda_1 \, P_iB_1+\cdots+\lambda_k \, P_iB_k) \, \bx_i = \zero,\quad i=1,\dots,k,
 \end{equation}
is nonsingular, then \eqref{eq:mepkalg}
has $n^k$ eigenvalues, but only
$\binom{n+k-1}{k}$ of them 
are 
solutions of \eqref{shap}.
On the other hand, if $\blambda=(\lambda_1,\dots,\lambda_k)$ is an eigenvalue of \eqref{shap} with an eigenvector $\bx$, then
$\blambda$ is an eigenvalue of \eqref{eq:mepkalg}, $\bz:= \bx \otimes \cdots \otimes \bx$ is the corresponding eigenvector, and
\ch{$\Delta_0^{-1}\Delta_i \bz = \lambda_i \bz$ for $i= 1,\dots,k$.}
Based on \Cref{mepPB} we derive the following numerical algorithm.\eject

\noindent\vrule height 0pt depth 0.5pt width \textwidth \\[-0.5mm]
{\bf Algorithm~1: Eigenvalues of a linear rectangular $k$-parameter MEP \eqref{shap}} \\[-3mm]
\vrule height 0pt depth 0.3pt width \textwidth \\
{\bf Input:} $(n+k-1) \times n$ matrices $A$ and $B_1,\dots,B_k$; \ch{tolerance {\sf tol} (default: $10^{-10}$)} \\
{\bf Output:} $\binom{n+k-1}{k}$ eigenvalues $\blambda^{(j)} = (\lambda_1^{(j)},\dots,\lambda_k^{(j)})$ \\
\begin{tabular}{ll}
{\footnotesize 1:} & Select $n \times (n+k-1)$ matrices $P_1, \dots, P_k$ with orthonormal rows. \\
{\footnotesize 2:} & Solve the $k$-parameter eigenvalue problem \eqref{eq:mepkalg} involving $n \times n$ matrices and obtain $n^k$\\ & eigenvalues $\blambda^{(1)}, \dots, \blambda^{(n^k)}$.\\
{\footnotesize 3:} & Keep only the
eigenvalues $\blambda^{(j)}$ \ch{for which}\\
& $\sigma_{n}(A+\lambda^{(j)}_1 B_1+\cdots+\lambda^{(j)}_kB_k)< {\sf tol} \cdot \|A+\lambda^{(j)}_1 B_1+\cdots+\lambda^{(j)}_kB_k\|$.
\end{tabular} \\[0.5mm]
\vrule height 0pt depth 0.5pt width \textwidth

In Algorithm~1, the projection matrices $P_1, \dots, P_k$ may be deterministic (provided as argument to the method), or chosen randomly.
For moderate values of $n^k$ we can solve the MEP \eqref{eq:mepkalg} in Step~2 by solving the corresponding system \eqref{drugi} of
GEPs with matrices $\Delta_0,\dots,\Delta_k$ of size
$n^k\times n^k$.
Alternatively, if $n^k$ is large and $k=2$ or $k=3$, then we can
apply an iterative subspace method (see, e.g., \cite{HKP, MeerP, HMMP}), where we do not have to form the $\Delta$-matrices explicitly,
to compute a subset of eigenvalues close to a given target.
For large values of $n^k$, in particular for large $k$, a method that might solve the MEP \eqref{eq:mepkalg} is
a homotopy method \cite{Fiber}.
For a large $k$ only a modest portion of the solutions to \eqref{eq:mepkalg} also is a solution to \eqref{shap}.
We explore this \ch{behavior} in \Cref{ex:nkrand}.

We remark that the complexity of the rank test in Step~3 is negligible compared to Step~2. If we need a nonzero
vector $\bx$ from \eqref{shap}, we can obtain it as a side result in Step~3 by taking the right singular
vector of $A+\lambda^{(j)}_1 B_1+\cdots+\lambda^{(j)}_kB_k$ that corresponds to the smallest singular value.

\begin{example} \label[example]{ex:nkrand} \rm
The next table shows numerical experiments for various values of $n$ and $k$, where matrices
$A$, $B_1$, $\dots$, $B_k$ are selected as random matrices using {\tt randn(n+k-1,n)} in Matlab. We use the
function {\tt multipareig} from package {\tt MultiParEig} \cite{MultiParEig} to solve the MEP in Step~2.

\begin{table}[htb!]
\centering
\caption{Results from Algorithm~1 \ch{on \eqref{shap}} with random $(n+k-1) \times n$ matrices.}
\vspace{1mm}
{\footnotesize \begin{tabular}{cr|cc|cc} \hline \rule{0pt}{2.2ex}%
$n$ & $k$ & Size $\Delta$ \ch{= \# Eigs \eqref{eq:mepkalg}} & \# Eigs \eqref{shap} & Max.~$\sigma_{n}$(eigs) & Min.~$\sigma_{n}$(non-eigs) \\[0.5mm]
\hline \rule{0pt}{2.3ex}%
$10$ & $\ph{1}2$ & $\ph{2}100$ & $\ph{12}55$ & $3.1\cdot 10^{-13}$ & $5.5\cdot 10^{-3}$ \\
$50$ & $\ph{1}2$ & $2500$ & $1275$ & $1.2\cdot 10^{-11}$ & $5.0\cdot 10^{-4}$ \\
$10$ & $\ph{1}3$ & $1000$ & $\ph{1}220$ & $5.0\cdot 10^{-11}$ & $4.4\cdot 10^{-3}$ \\
$\ph{1}7$ & $\ph{1}4$ & $2401$ & $\ph{1}210$ & $2.4\cdot 10^{-11}$ & $1.7\cdot 10^{-3}$ \\
$\ph{1}4$ & $\ph{1}5$ & $1024$ & $\ph{12}56$ & $4.1\cdot 10^{-12}$ & $2.4\cdot 10^{-2}$ \\
$\ph{1}2$ & $10$ & $1024$ & $\ph{12}11$ & $6.2\cdot 10^{-10}$ & $7.8\cdot 10^{-3}$ \\ \hline
\end{tabular}}
\end{table}

The maximal value of
$\sigma_n(A+\lambda^{(j)}_1 B_1+\cdots+\lambda^{(j)}_kB_k)$ for all eigenvalues $\blambda^{(j)}$ that satisfy the rank drop in step~3 is displayed in the fifth column.
Similarly, in the last column is the
minimal singular value for all eigenvalues that do not satisfy the criterion.
We can see a clear gap between the eigenvalues of \eqref{eq:mepkalg} that are also eigenvalues
of \eqref{shap} and those that are not.

The ratio between the number of solutions
of \eqref{shap} and \eqref{eq:mepkalg} decreases
 with increasing $k$ and only a small portion of eigenvalues of \eqref{eq:mepkalg} are
also eigenvalues of \eqref{shap}. The transformation of \eqref{shap} into \eqref{eq:mepkalg} can thus dramatically increase the
size of the problem whose solutions are
candidates for the solutions of \eqref{shap}.
As we will see
in \Cref{subs:Macaulay}, the problem \ch{\eqref{shap} grows into an even larger one}
if we \ch{use} a block Macaulay matrix. In the following subsection,
we show how we can reduce the $\Delta$-matrices to the
size equal to the number of eigenvalues of \eqref{shap} \ch{using the results from \cite{ALS}}.
\end{example}

\subsection{Compression} \label{subs:compress}
An elegant observation from Alsubaie \cite{ALS} that helps to reduce the $\Delta$-matrices
\eqref{drugi} related to the MEP
\eqref{eq:mepkalg} is that
vectors of the form $\bx \otimes\cdots\otimes \bx$ span a subspace ${\cal T}$ of dimension
$\binom{n+k-1}{k}$ in $\C^n\otimes \cdots \otimes \C^n$ \ch{of dimension $n^k$}. Indeed, if
$\bx= [x_1, \, \dots, \, x_n]^T$, then the
elements of $\bz = \bx \otimes \cdots \otimes \bx$ are $z_{i_1\cdots i_k}=x_{i_1}\cdots x_{i_k}$ for $i_1,\dots,i_k= 1,\dots,n$.
If $(\ell_1,\dots,\ell_k)$ is a permutation of $(i_1,\dots,i_k)$, then $z_{\ell_1\cdots \ell_k}=z_{i_1\cdots i_k}$. There are
$\binom{n+k-1}{k}$ multi-indices $(i_1,\dots,i_k)$, where $1 \le i_1 \le \cdots \le i_k\le n$,
and all remaining multi-indices
are their permutations.
Let $r(i_1,\dots,i_k) =(i_1-1)n^{k-1}+\cdots+(i_{k-1}-1)n+i_k$ be an enumeration from a
multi-index $(i_1,\dots,i_k)$ into
a single index from $1$ to $n^k$, and let $c(j_1,\dots,j_k)$ be a transformation
from a set of multi-indices such that $1 \le j_1 \le \cdots \le j_k\le n$ into a single index from $1$ to $\binom{n+k-1}{k}$. We can define an $n^k\times \binom{n+k-1}{k}$ matrix $T$ such that
$T_{pq}= 1$ if $p=r(i_1,\dots,i_k)$, $q=c(j_1,\dots,j_k)$ and $(i_1,\dots,i_k)$ is a permutation of $(j_1,\dots,j_k)$, and
$T_{pq}= 0$ otherwise. Clearly, each row of $T$ has exactly one nonzero element $1$ and each column of $T$ contains at least one
and at most $k!$ nonzero elements. Then, each vector $\bz$ from ${\cal T}$ can be written as $\bz =T \bw$ for a
$\bw\in \C^{\binom{n+k-1}{k}}$. 

The matrices $\Delta_0^{-1}\Delta_i$ for $i= 1,\dots,k$ commute and ${\cal T}$ is their common invariant subspace. It follows that
$\Delta_0^{-1}\Delta_i T=  T G_i$, \ch{where
$G_i$ is the restriction of $\Delta_0^{-1}\Delta_i$ to the invariant subspace ${\cal T}$,}
for $i= 1,\dots,k$,  and  matrices
$G_1,\dots,G_k$ of size
$\binom{n+k-1}{k}\times \binom{n+k-1}{k}$ commute. Also, if $\blambda$ is an eigenvalue of \eqref{shap}, then
$\blambda$ is a common eigenvalue of $G_1,\dots,G_k$, \ch{i.e., there
exists a nonzero vector $\bw$ such that $G_i\bw =\lambda_i \bw$ for $i=1,\ldots,k$.}

For a generic matrix $Q$ of size $\binom{n+k-1}{k}\times n^k$, the matrix $Q\Delta_0 T$ is nonsingular since $\Delta_0$ is nonsingular
and $T$ has full rank. If we define
$\widehat\Delta_i = Q\Delta_iT$ for $i = 0,\dots,k$, we get matrices of size $\binom{n+k-1}{k}\times \binom{n+k-1}{k}$ such that
$\widehat\Delta_0$ is nonsingular. It follows from ${\widehat\Delta_0}^{-1}\widehat\Delta_i=G_i$ that the matrices
${\widehat\Delta_0}^{-1}\widehat\Delta_i$ for $i= 1,\dots,k$ commute and we can obtain solutions of \eqref{shap} from these matrices.
The size of the matrices $\widehat \Delta_i$ is optimal as it matches the number of solutions of \eqref{shap}.

Now, for $i = 1,\dots,k$, we introduce the $(n+k-1)^k\times n^k$ matrices \[
\wt \Delta_0 =
\left| \begin{matrix} B_1 & \cdots & B_k \\[-0.5mm]
\vdots & & \vdots \\[-0.5mm]
B_1 & \cdots & B_k
 \end{matrix} \right|_\otimes, \qquad
\wt \Delta_i = -\ \left|\begin{matrix}B_1 & \cdots & B_{i-1} & A & B_{i+1} & \cdots & B_k \\[-0.5mm]
\vdots & & \vdots & \vdots & \vdots & & \vdots \\[-0.5mm]
B_1 & \cdots & B_{i-1} & A & B_{i+1} & \cdots & B_k\end{matrix}\right|_\otimes.
\]
If $P_1, \dots, P_k$ are the $n \times (n+k-1)$ matrices from \Cref{mepPB}, then
the $\Delta$-matrices of the corresponding \eqref{drugi} are
$\Delta_i =(P_1 \otimes \cdots \otimes P_k)\,\wt \Delta_i$ for $i = 0,\dots,k$. If the MEP
\eqref{eq:mepkalg} is nonsingular, then 
$\wt \Delta_0$ has full rank.
Then $\widehat\Delta_i =Q\,(P_1 \otimes \cdots \otimes P_k)\,\wt\Delta_i\,T$ for $i = 0,\dots,k$. Instead of
selecting random matrices $Q$, $P_1$, \dots, $P_k$, we get the same effect if we form
$\widehat\Delta_i = \widehat L \, \wt\Delta_i \, T$ for $i = 0,\dots,n$ using
a generic matrix
$\widehat L$ of size $\binom{n+k-1}{k}\times (n+k-1)^k$.
Even better, based on the following result from \cite{ALS} we can construct
a particular sparse matrix $L$ such that $L \, \wt \Delta_0 \, T$ is nonsingular for generic matrices
$B_1,\dots,B_k$.

The matrix $\wt\Delta_0 \, T$ of size $(n+k-1)^k\times \binom{n+k-1}{k}$ has full rank. The next lemma explains
how we can select $\binom{n+k-1}{k}$ linearly independent rows. As the proof is technical but not difficult, we will omit it; for the details, see \cite[Sec.~4.2.2]{ALS}.

\begin{lemma} \label[lemma]{lem:rows_of_D0}
Let $\Psi: = \wt \Delta_0 \, T$ be a $(n+k-1)^k\times \binom{n+k-1}{k}$ matrix as above
and let $\Psi(i_1\dots i_k,:)$ denote the row of $\Psi$ corresponding to a multi-index
$(i_1,\dots,i_k)$, where $1 \le i_1,\dots,i_k\le n+k-1$. Then:
\vspace{-2mm}
\begin{enumerate} \itemsep=-1mm
\item[a)] If $i_p=i_q$ for $p\ne q$, then $\Psi(i_1\dots i_k,:) = 0$.
\item[b)] If $(j_1,\dots,j_k) =(\sigma(i_1),\dots,\sigma(i_k))$ for a permutation $\sigma$, then \newline
$\Psi(j_1\dots j_k,:) ={\rm sgn}(\sigma) \, \Psi(i_1\dots i_k,:)$.
\end{enumerate}
\end{lemma}

By \Cref{lem:rows_of_D0}, a row of $\wt \Delta_0 \, T$ with a multi-index $(i_1,\dots,i_k)$
is zero if two indices are equal. Moreover, it agrees up to a sign to the row with multi-index $(\sigma(i_1),\dots,\sigma(i_k))$,
where permutation $\sigma$ sorts indices in increasing order.
Since $\wt\Delta_0 \, T$ has rank $\binom{n+k-1}{k}$ and there are exactly $\binom{n+k-1}{k}$
strictly ordered multi-indices $(i_1,\dots,i_k)$
such that $1 \le i_1<\cdots<i_k\le n+k-1$, it follows that if we select all rows of $\wt \Delta_0 \, T$ with these indices,
then the obtained matrix is square and nonsingular.

Similarly to before we introduce
a transformation $\gamma(i_1,\dots,i_k)$ from
the set of strictly ordered multi-indices into a single index from 1 to $\binom{n+k-1}{k}$ and
a transformation $\beta(i_1,\dots,i_k) =(i_1-1)(n+k-1)^{k-1}+\cdots +(i_{k-1}-1)(n+k-1)+i_k$ from a
multi-index $(i_1,\dots,i_k)$, where $1 \le i_1,\dots,i_k\le n+k-1$, into
a single index from $1$ to $(n+k-1)^k$. Then define the $\binom{n+k-1}{k} \times (n+k-1)^k$ matrix $L$ by
$L_{pq}= 1$ if $\beta(i_1,\dots,i_k) =q$, $1 \le i_1<\cdots<i_k\le n+k-1$, and $\gamma(i_1,\dots,i_k) =p$; otherwise,
$L_{pq}= 0$.

\begin{example} \label[example]{ex:exTL}
For $n=2$ and $k=2$ we get the following matrices $T$ and $L$, which compress the $\wt \Delta$-matrices from size $9 \times 4$ to size $3 \times 3$: 
\[
T= \smath{\begin{blockarray}{cccc}
& 11 & 12 & 22 \\
\begin{block}{c[ccc]}
 11 & 1 & 0 & 0  \\
 12 & 0 & 1 & 0 \\
 21 & 0 & 1 & 0  \\
 22 & 0 & 0 & 1 \\
\end{block}
\end{blockarray}},
\qquad
L= \smath{\begin{blockarray}{cccccccccc}
& 11 & 12 & 13 & 21 & 22 & 23 & 31 & 32 & 33 \\
\begin{block}{c[ccccccccc]}
 12 & 0 & 1 & 0 & 0 & 0 & 0 & 0 & 0 & 0 \\
 13 & 0 & 0 & 1 & 0 & 0 & 0 & 0 & 0 & 0 \\
 23 & 0 & 0 & 0 & 0 & 0 & 1 & 0 & 0 & 0 \\
\end{block}
\end{blockarray}}.
\]
\chh{Note that we label rows of $T$ by multiindices $(i,j)$ such that $1\le i,j\le n=2$ and
columns by multiindices $(p,q)$ such that $1\le p\le q\le n=2$ using standard lexicographic ordering.
There are ones in $T$ exactly on places where $(i,j)$ is a permutation of $(p,q)$.
In matrix $L$ we label rows by multiindices $(i,j)$ such that $1\le i<j\le n+k-1=3$ and columns
 by multiindices $(p,q)$ such that $1\le p,q\le n+k-1=3$, again using standard lexicographic ordering.
 There are ones in $L$ exactly on places where $(i,j)=(p,q)$.}
\end{example}

\begin{theorem} \label{thm:ALS}
There exists a generic set $\Omega\subset (\C^{(n+k-1) \times n})^k$
such that for each
$A\in \C^{(n+k-1) \times n}$ and $(B_1,\ldots,B_k)\in\Omega$ the following is true:
\vspace{-2mm}
\begin{enumerate} \itemsep=-1mm
\item[1)] The matrix $D_0 = L \, \wt\Delta_0 \, T$ is nonsingular.
\item[2)] Problem \eqref{shap} has $\binom{n+k-1}{k}$ eigenvalues counting multiplicities.
\item[3)] ${D_0}^{-1}D_1,\dots,{D_0}^{-1}D_k$, where
$D_i =L \, \wt\Delta_i \, T$ for $i = 1,\dots,k$, commute.
\item[4)] $\blambda=(\lambda_1,\dots,\lambda_k) \in \C^k$ is a common eigenvalue of
${D_0}^{-1}D_1,\dots,{D_0}^{-1}D_k$
if and only if $\blambda$ is an eigenvalue of \eqref{shap}.
\end{enumerate}
\end{theorem}

\begin{proof}
For 1) and 2) we consider $\det(D_0)$ as a polynomial in elements of $B_1,\dots,B_k$. Then $\det(D_0) = 0$ is a proper algebraic \ch{variety},
since, if we take $B_i = P_i^T$, \ch{where
$P_i = [0_{n,i-1} \ \, I_n \ \, 0_{n,k-i}],\ i=1,\ldots,k$,
are matrices from the proof of \Cref{mepPB}}, then
$\det(D_0) \ne 0$. Therefore, there exists a
generic set $\Omega_1\subset \left(\C^{(n+k-1) \times n}\right)^k$ such that $D_0$ is nonsingular for
$(B_1,\dots,B_k) \in \Omega_1$. We know from \ch{\Cref{mepPB}} and \Cref{lem:shapiro} that there exists
a generic set $\Omega_2\subset \left(\C^{(n+k-1) \times n}\right)^k$ such that for $(B_1,\dots,B_k) \in \Omega_2$ and
arbitrary $A$ problem \eqref{shap} has $\binom{n+k-1}{k}$ eigenvalues counting multiplicities.
It follows that 1) and 2) are true for $\Omega:= \Omega_1\cap\Omega_2$.

Items 3) and 4) follow from the discussion before \Cref{lem:rows_of_D0}.
\end{proof}

We remark that the results in \Cref{thm:ALS} are closely related to results in \cite[Sec.~4]{ALS}. Using a  different derivation Alsubaie obtained the same matrices $D_0,\ldots,D_k$ using compression matrices
$L$ and $T$, and derived a connection between the
solutions of \eqref{shap} and the GEP $D_1\bw =\lambda_1 D_0 \bw$. \ch{The results in  \Cref{thm:ALS} are
more general and connect eigenvalues of \eqref{shap} to common eigenvalues of
commuting matrices ${D_0}^{-1}D_1,\dots,{D_0}^{-1}D_k$ in a similar way as nonsingular standard MEP \eqref{problem} is related to commuting matrices
${\Delta_0}^{-1}\Delta_1,\dots,{\Delta_0}^{-1}\Delta_k$ from \eqref{drugi}. This enables us to apply
several theoretical results as well as numerical methods for
systems of joint GEPs that are available for the standard MEPs.}

Based on \Cref{thm:ALS} and inspired by Algorithms 2 and 3 in Alsubaie \cite{ALS} we propose the following algorithm.
While the algorithms in \cite{ALS} solve $D_1\bw =\lambda_1 D_0 \bw$ under the assumption that $\lambda_1$ is simple,
we exploit the existing \ch{numerical} methods from \cite{MultiParEig} to solve a joint system of GEPs and, therefore, the method
works for the general case \ch{with possible multiple eigenvalues}.

\noindent\vrule height 0pt depth 0.5pt width \textwidth \\[-0.5mm]
{\bf Algorithm~2: Eigenvalues of a rectangular linear $k$-parameter MEP \eqref{shap}} \\[-3mm]
\vrule height 0pt depth 0.3pt width \textwidth \\
{\bf Input:} $(n+k-1) \times n$ matrices $A$ and $B_1,\dots,B_k$ \\
{\bf Output:} eigenvalues $\blambda^{(j)} = (\lambda_1^{(j)},\dots,\lambda_k^{(j)})$ for $j=1,\dots,\binom{n+k-1}{k}$\\
\begin{tabular}{ll}
{\footnotesize 1:} & Construct $\binom{n+k-1}{k}\times (n+k-1)^k$ matrix $L$ and $n^k\times \binom{n+k-1}{k}$ matrix $T$ for the compression.\\
{\footnotesize 2:} & Compute $(n+k-1)^k\times n^k$ matrices $\wt \Delta_0,\dots, \wt \Delta_k$. \\
{\footnotesize 3:} & Compute matrices $D_i = L \, \wt \Delta_i \, T$ for $i = 0,\dots, k$. \\
{\footnotesize 4:} & Solve a joint system of GEPs
$D_i \bw = \lambda_i D_0 \bw$ for $i = 1,\dots,k$ and obtain eigenvalues $\blambda^{(j)}$\\
& for $j=1,\dots,\binom{n+k-1}{k}$.\\
\end{tabular} \\[0.5mm]
\vrule height 0pt depth 0.5pt width \textwidth

The sparsity of $L$ and $T$ can be exploited to compute the
$D_i$ without explicitly computing the $\wt \Delta_i$. Note that the $\wt \Delta_i$ can be so
large that it is impossible to compute the matrix due to memory limitations, while $D_i$ can be much smaller.
In \cite{ALS}, two methods \ch{are} presented for the computation of the $D_i$: one for problems of moderate
size where it is more efficient to first explicitly compute the $\wt \Delta_i$, and another one for larger problems where the required elements of the $\wt \Delta_i$ are computed one by one.

In Algorithm~2 there are no redundant solutions, and, therefore, it is more efficient than Algorithm~1 for
problems where $\binom{n+k-1}{k}$ is small enough that we can explicitly compute
the $\widetilde\Delta$-matrices and solve the system of GEPs in Step~4; recall that the $\Delta$-matrices in Algorithm~1 are of much larger size $n^k\times n^k$.

If $\binom{n+k-1}{k}$ is too large for the above, \ch{then} an option is to apply an iterative subspace method (for instance {\tt eigs} in Matlab) to
compute just a small subset of eigenvalues close to a given target. As we still have to compute the matrices $D_i$, this
is suitable 
for problems with sparse matrices where we can exploit the sparsity of $L$ and $T$.
If this is not the case, in particular for $k=2$ and $k=3$, it might be more efficient
to use Algorithm~1 and apply an iterative subspace method tailored to the MEP, that does not need the $\Delta$-matrices explicitly.

\Cref{tab:compare} provides an overview of the main advantages and disadvantages of Algorithms~1 and 2.

\begin{table}[htb!]
\centering
\caption{Summary of pros and cons of Algorithms~1 and 2.} \label{tab:compare}
\begin{tabular}{lll} \hline \rule{0pt}{2.3ex}%
& Alg.~1 & Alg.~2 \\ \hline \rule{0pt}{2.3ex}%
Pros & $\bullet$ Simple construction & $\bullet$ Uses $D$-matrices of optimal size \\
& $\bullet$ Numerical methods for MEPs that do & $\bullet$ No redundant solutions \\
& \phantom{- }not require $\Delta$-matrices can be applied & $\bullet$ Sparsity is preserved \\[0.5mm] \hdashline \rule{0pt}{2.8ex}%
Cons & $\bullet$ Uses much larger $\Delta$-matrices& $\bullet$ $D$-matrices are needed explicitly \\
& $\bullet$ Sparsity is lost with random $P$ & $\bullet$ Kronecker structure of $\wt \Delta$-matrices\\
&  & \phantom{- }is lost in multiplication by $L$ and $T$\\\hline
\end{tabular}
\end{table}

\subsection{Block Macaulay algorithm} \label{subs:Macaulay}
To the best of our knowledge, the only other \ch{available}
numerical method for \ch{solving a linear}
RMEP \eqref{rectmep} uses block Macaulay matrices.
We give a brief introduction of the method, for details see, e.g., \cite{DeMoor_LTI, DeMoor_CIS,DeMoor_ARMA, DeMoor_LAA}.
In case $k=2$, we start from
\eqref{shap} and multiply the equation with monomials $\lambda_1^i\lambda_2^j$ of increasing order,
where we add rows in blocks of the same degree. We get a homogeneous system with
the block Macaulay matrix
\begin{equation} \label{eq:macaulay}
\begin{blockarray}{ccccccccccc}
& \sca{1} & \sca{\lambda_1} & \sca{\lambda_2} & \sca{\lambda_1^2} & \sca{\lambda_1\lambda_2} & \sca{\lambda_2^2} & \sca{\lambda_1^3} & \sca{\lambda_1^2\lambda_2} & \dots \\
\begin{block}{c[cccccccccc]}
\sca{1} & A & B_1 & B_2 & 0 & 0 & 0 & 0 & 0 & \cdots \\
\sca{\lambda_1} & 0 & A & 0 & B_1 & B_2 & 0 & 0 & 0 & \cdots \\
\sca{\lambda_2} & 0 & 0 & A & 0 & B_1 & B_2 & 0 & 0 & \cdots \\
\sca{\lambda_1^2} & 0 & 0 & 0 & A & 0 & 0 & B_1 & B_2 & \cdots \\
\vdots & \vdots & \vdots & \vdots& \vdots & \ddots & \ddots& \ddots & \ddots & \ddots \\
\end{block}
\end{blockarray}
\mtxb{\bx \\ \lambda_1 \bx \\ \lambda_2 \bx \\ \lambda_1^2 \bx \\
\vdots} = \zero
\end{equation}
and a solution to \eqref{shap} corresponds to a structured vector in the nullspace that has a block Vandermonde-like structure.
We say that a block Macaulay matrix is of degree $m$ if it includes all columns corresponding to monomials of degree at most $m$.
A block Macaulay matrix of \chh{degree} $m$ for \eqref{shap} then contains rows corresponding to monomials of \ch{total} degree $m-1$ or less.
For a degree $m\ge m^*$, \ch{where $m^*$ is the minimal sufficient degree, the structure of the nullspace stabilizes and} we can compute
all solutions of \eqref{shap} from the nullspace of the block Macaulay matrix; for details see, e.g., \cite{DeMoor_ARMA}.

\begin{lemma} \label[lemma]{lem:maxsize}
(\cite[p.~185]{DeMoor_LAA}) A block Macaulay matrix of degree $m$ for the linear rectangular $k$-parameter eigenvalue problem \eqref{shap} with
matrices of size $(n+k-1) \times n$ has size
\begin{equation} \label{eq:binom_Mac}
\textstyle (n+k-1)\,\binom{m+k-1}{k} \ \times \ n \, \binom{m+k}{k}.
\end{equation}
\end{lemma}
\begin{proof} The result follows from the fact that there are $\binom{m+k}{k}$ different monomials in
$k$ variables of \ch{total} degree up to $m$.
\end{proof}

We have implemented the nullspace block Macaulay algorithm
from \cite{DeMoor_ARMA}. Based on numerous
numerical experiments with various $k$ and $n$, we
conjecture that $n$ is the minimal sufficient degree for the extraction of eigenvalues of \eqref{shap} using a linear polynomial for a shift.
With this hypothesis we can assume that $m=n$ in \eqref{eq:binom_Mac},
compute the minimal size of the block Macaulay matrix and compare that to
Algorithm~1, where the $\Delta$-matrices of \eqref{eq:mepkalg}
are of size $n^k\times n^k$, and Algorithm~2, where the compressed matrices $D_i$
are of size $\binom{n+k-1}{k}\times \binom{n+k-1}{k}$. A comparison for $2$, $3$, and $4$ parameters
and various values of $n$ is given in \Cref{tab:comp_rmped_234}.
\chh{Let us remark that the very recent recursive solver in \cite{DeMoor_SISC} does not require an explicit construction of the block Macaulay matrices, which
speeds us the computation tremendously and enables solution of larger problems. However, even with this speed up, due to large dimensions of the block
Macaulay matrices, this solver may be slower than the methods presented in this paper; see, e.g., \Cref{ex:arma11} and \Cref{ex:lti2_11}.}

\begin{table}[htb!]
\centering
{\footnotesize
\caption{Sizes of matrices of the block Macaulay method, Algorithm~1, and Algorithm~2 required to solve a generic linear
$k$-parameter RMEP \eqref{shap} with matrices of size $(n+k-1) \times n$ for $k=2,3,4$ and
$n=2,4,\dots,20$. For rectangular matrices from the block Macaulay method we give the number of rows, which is smaller than the number of columns.} \label{tab:comp_rmped_234}
\vspace{1mm}
\begin{tabular}{r|ccc|ccc|ccc} \hline \rule{0pt}{2.1ex}%
& \multicolumn{3}{c|}{$k=2$} & \multicolumn{3}{|c|}{$k=3$} & \multicolumn{3}{|c}{$k=4$} \\
$n$ & Mac. & Alg.~1 & Alg.~2 & Mac. & Alg.~1 & Alg.~2 & Mac. & Alg.~1 & Alg.~2
\\
\hline \rule{0pt}{2.2ex}%
\ph{1}2 & $\ph{111}9$ & $\ph{12}4$ & $\ph{12}3$ & $\ph{111}16$ & $\ph{123}8$ & $\ph{123}4$ & $\ph{1111}25$ & $\ph{1234}16$ & $\ph{123}5$ \\
\ph{1}4 & $\ph{11}50$ & $\ph{1}16$ & $\ph{1}10$ & $\ph{11}120$ & $\ph{12}64$ & $\ph{12}20$ & $\ph{111}245$ & $\ph{123}256$ & $\ph{12}35$ \\
\ph{1}6 & $\ph{1}147$ & $\ph{1}36$ & $\ph{1}21$ & $\ph{11}448$ & $\ph{1}216$ & $\ph{12}56$ & $\ph{11}1134$ & $\ph{12}1296$& $\ph{1}126$ \\
\ph{1}8 & $\ph{1}324$ & $\ph{1}64$ & $\ph{1}36$ & $\ph{1}1200$ & $\ph{1}512$ & $\ph{1}120$ & $\ph{11}3630$ & $\ph{12}4096$& $\ph{1}330$ \\
10 & $\ph{1}605$ & $100$ & $\ph{1}55$ & $\ph{1}2640$ & $1000$ & $\ph{1}220$ & $\ph{11}9295$ & $\ph{1}10000$ & $\ph{1}715$ \\
12 & $1014$ & $144$ & $\ph{1}78$ & $\ph{1}5096$ & $1728$ & $\ph{1}364$ & $\ph{1}20475$ & $\ph{1}20736$ & $1365$ \\
14 & $1575$ & $196$ & $105$ & $\ph{1}8960$ & $2744$ & $\ph{1}560$ & $\ph{1}40460$ & $\ph{1}38416$ & $2380$ \\
16 & $2312$ & $256$ & $136$ & $14688$ & $4096$ & $\ph{1}816$ & $\ph{1}73644$ & $\ph{1}65536$ & $3876$ \\
18 & $3249$ & $324$ & $171$ & $22800$ & $5832$ & $1140$ & $125685$ & $104976$ & $5985$ \\
20 & $4410$ & $400$ & $210$ & $33880$ & $8000$ & $1540$ & $203665$ & $160000$ & $8855$ \\ \hline
\end{tabular}}
\end{table}

Of the three methods, Algorithm~2 uses the smallest matrices of the optimal size $\binom{n+k-1}{k}$ to solve a linear RMEP \eqref{shap}.
We can see from Lemma~\ref{lem:maxsize} that for the same problem
we need block Macaulay matrices with
$(n+k-1)$-times as many rows.
Algorithm~1, which uses matrices of size $n^k$, is also asymptotically (for a fixed $k$)
more efficient than
the block Macaulay method.
\ch{However, for $k\ge 4$ is Algorithm~1 more efficient only when $n$ is very large, which makes
Algorithm~1 in practice feasible and efficient  only for $k=2$ and $k=3$.
}

\section{Polynomial rectangular MEPs} \label{sec:polyrmep}
A polynomial RMEP has the form \eqref{rectmep},
where the highest total degree of the monomials is $d>1$ and matrices are of size $(n+k-1) \times n$.
We assume that the normal rank of $M$ is full, {\ch i.e.}, ${\rm nrank}(M)=n$.
The following new theorem extends \ch{the} result from \cite{Shapiro} to polynomial RMEPs.

\ch{
\begin{theorem} \label{lem:fin_sol} A
generic polynomial rectangular $k$-parameter eigenvalue problem \eqref{rectmep} of degree $d$ has exactly
\begin{equation} \label{eq:inters_number}
d^k \, \binom{n+k-1}{k}
\end{equation}
eigenvalues, which are all finite.
\end{theorem}
}

\begin{proof}
We consider a linear space $\calm$ of all complex $(n+k-1) \times n$ matrices. \ch{ It is a projective variety.}
Let $\calm_1$ be a subset of $\calm$ that contains matrices that do not have full rank, i.e., rank is less than $n$.
Then $\calm_1$ is \ch{a projective} variety of codimension $k$ and degree $\binom{n+k-1}{k}$; see, e.g., \cite[Prop.~2.15]{detrings}.

\ch{
Each $\ell$-tuple $\underline{A}=(A_\bomega)_{|\bomega|\le d}\in{\cal M}^\ell$, where $\ell=\binom{k+d}{k}$ is the number of monomials in $k$ variables of total degree $d$ or less,
corresponds to a polynomial RMEP \eqref{rectmep}. For each $\underline{A}\in{\cal M}^\ell$ we
define a subset of $\calm$ of the form
\begin{equation}\label{eq:s_a_variety}
\cals(\underline A):= \Big\{\sum_{ |{^h}\bomega| = d} {^h}\blambda^{{^h}\bomega} A_\bomega:\ {^h}\blambda\in \PP^k(\C)\Big\},
\end{equation}
which corresponds to the image of the homogenization of the polynomial $M({\blambda})$ from \eqref{rectmep}.
Here $\PP^k(\C)$ is projective space of dimension $k$, ${^h}\blambda=\left[\lambda_0:\lambda_1:\cdots:\lambda_k\right]$ are homogeneous coordinates and
${^h}\bomega=(d-|\bomega|,\omega_1,\ldots,\omega_k)$ is the corresponding homogenized multi-index.}
For a generic $\underline A$, \chh{$\zero\not\in\cals(\underline A)$
and it follows that}
\ch{$\cals(\underline A)$} is
\ch{a projective} variety of dimension $k$ \ch{\cite[Ch. 8, \S{5}, Thm. \chh{11}]{CLO}.}
\ch{The} degree \ch{of $\cals(\underline A)$} is equal to the number of intersection points with $k$ generic hyperplanes in
$\calm$. It is easy to see that intersections with hyperplanes lead to a system of \ch{homogeneous} polynomial equations
$s_i(\ch{{^h}\blambda}) = 0$, $i = 1,\dots,k$, where $s_i$ is a scalar \ch{$(k+1)$}-variate
\ch{homogeneous} polynomial of degree $d$.
As such system of polynomial equations generically has $d^k$ solutions by B\'ezout's Theorem (see, e.g., \cite[Ch. 4, Sec. 2] {invalggeo}),
this is the degree of \ch{$\cals(\underline A)$.}

If \ch{$\cals(\underline A)$} is transversal to $\calm_1$, i.e., $\ch{\cals(\underline A)}\cap \calm_1$ is a finite set, then,
by B\'ezout's Theorem, $|\ch{\cals(\underline A)}\cap \calm_1| = \deg(\ch{\cals(\underline A)}) \cdot \deg(\calm_1)=d^k \, \binom{n+k-1}{k}$.
\chh{What remains to show is that there exists $\underline A\in\calm^\ell$ such that
$|\ch{\cals(\underline A)}\cap \calm_1|=d^k \, \binom{n+k-1}{k}$. Namely, all
$\underline A\in\calm^\ell$ such that \ch{$\cals(\underline A)$} is not transversal to $\calm_1$ form
a subvariety ${\cal U}$ of $\calm^\ell$.}
To show that ${\cal U}$ is a proper subvariety, we take a polynomial RMEP
 \eqref{rectmep}
that contains only monomials $1,\lambda_1^d,\dots,\lambda_k^d$ and write it as a linear
RMEP using substitution $\bxi=(\xi_1,\dots,\xi_k) =(\lambda_1^d,\dots,\lambda_k^d)$.
We know from \Cref{sec:linrmep} that a generic linear RMEP has $\binom{n+k-1}{k}$ eigenvalues $\bxi$ and from each $\bxi$ we get $d^k$ eigenvalues
$(\lambda_1,\dots,\lambda_k)$. This gives
$d^k \, \binom{n+k-1}{k}$ eigenvalues and proves that \eqref{eq:inters_number} is the number of eigenvalues \ch{(finite and infinite)} for a generic polynomial
 RMEP of degree $d$.

\ch{Finally, let us show that for a generic polynomial RMEP, i.e., for a generic $\underline A\in{\cal M}^\ell$,  all eigenvalues are finite.
Similarly to \eqref{eq:s_a_variety} we define
 \[
\cals_d(\underline A):= \Big\{\sum_{ |\bomega| = d} \blambda^{\bomega} A_\bomega:\ \zero\ne\blambda\in \C^{k}\Big\},
\]
where we consider only monomials of total degree equal to $d$ in \eqref{rectmep} and nonzero $\blambda$. Because of
$\blambda\ne\zero$ we can view $\cals_d(\underline A)$ as a projective variety of dimension $k-1$ and, therefore,
since the codimension of $\calm_1$ is $k$,
$\chh{\cals_d}(\underline A)\cap \calm_1$ is empty for a generic $\underline A$. \chh{From ${\cals_d}(\underline A)\subset
\cals(\underline A)$ it follows that}
for a generic $\underline A$ is then
the component $\lambda_0$ of ${^h}\blambda$
nonzero for each element of $\cals(\underline A)\cap \calm_1$ and thus
all eigenvalues $\blambda$ are finite.
}
\end{proof}

\section{Numerical methods for polynomial rectangular MEPs} \label{subs:num_polyrect}
\ch{The only numerical approach for a polynomial RMEP that we are aware of uses block Macaulay matrices.
The version for linear RMEPs is presented in \Cref{subs:Macaulay}, but the method was actually designed for
a polynomial RMEP of arbitrary degree. A generalization from linear to the polynomial case is straightforward, for details, see, e.g., \cite{DeMoor_ARMA, DeMoor_LAA}}.
In this section we
present several \ch{new} options to solve a polynomial RMEP numerically.
All methods have in common that
we get solutions of \eqref{rectmep} from eigenvalues of
a joint set of GEPs. For a general polynomial RMEP, all methods lead to singular GEPs
of size larger than the number of solutions of \eqref{rectmep}. One way to solve such problems is by using \ch{the} staircase algorithm from \cite{MP2}.

It is involved to formulate an algorithm for a general polynomial RMEP, as it depends on the structure of the monomials in \eqref{rectmep}.
Instead, we will demonstrate the methods by considering a generic quadratic R2EP (that is, $k=2$ and $d=2$) of the form
\begin{equation}
\label{eq:quadrecttwomep}
(A_{00}+\lambda A_{10} +\mu A_{01} + \lambda^2 A_{20} + \lambda \mu A_{11}+ \mu^2 A_{02})\,\bx = \zero,
\end{equation}
where the $A_{ij}$ are $(n+1) \times n$ matrices.

\subsection{Transformation to a quadratic two-parameter eigenvalue problem} \label{subs:trans_q2ep}
A first option is to use a similar approach as in Algorithm~1.
If we multiply \eqref{eq:quadrecttwomep} by random matrices $P_1$ and $P_2$ of size $n \times(n+1)$, we obtain
a \ch{standard} quadratic 2EP with $n \times n$ matrices
\begin{equation}
\label{eq:quadprojrecttwomep}
(P_iA_{00}+\lambda P_iA_{10} +\mu P_iA_{01} + \lambda^2 P_iA_{20} + \lambda \mu P_iA_{11}+ \mu^2 P_i A_{02})\,\bx_i = \zero,\quad i=1,2.
\end{equation}
The problem \eqref{eq:quadprojrecttwomep} generically has
$4n^2$ eigenvalues that include \ch{the} $2n(n+1)$ solutions
of \eqref{eq:quadrecttwomep}.

\ch{To solve \eqref{eq:quadprojrecttwomep}, we can apply any numerical method for standard quadratic MEPs,
see, e.g., \cite{MP2}.
One option is to} linearize \eqref{eq:quadprojrecttwomep} as a linear 2EP
with $3n \times 3n$ matrices, see, e.g., \cite{uniform, HMP}, for example as
\begin{equation}
\label{eq:linquadRMEP}
(V_{i0}+\lambda V_{i1} +\mu V_{i2})\,\bz_i = \zero,\quad i=1,2,
\end{equation}
where $\bz_i = [\bx_i^T, \, \lambda \bx_i^T, \, \mu \bx_i^T]^T$, and
\[
V_{i0}= \smtxa{ccc}{P_i A_{00} & P_i A_{10} & P_i A_{01} \\[0.5mm]
 0 & -I & 0  \\
 0 & 0 & -I}, \quad
V_{i1}= \smtxa{ccc}{0 & P_i A_{20} & P_i A_{11} \\[0.5mm]
I & 0 & 0 \\
0 & 0 & 0}, \quad
V_{i2}= \smtxa{ccc}{ 0 & 0  & P_i A_{02} \\
0 & 0 & 0  \\
I & 0 & 0 } \
\]
for $i = 1,2$. This gives
joint GEPs $\Delta_1 \bv = \lambda \, \Delta_0 \bv$ and
$\Delta_2 \bv = \mu \, \Delta_0 \bv$ of size $9n^2\times 9n^2$ that are both singular and have $4n^2$ eigenvalues that are
solutions of \eqref{eq:linquadRMEP} (and include \ch{the} $2n(n+1)$ solutions
of \eqref{eq:quadrecttwomep}).
An alternative 
is to
apply a Jacobi--Davidson method for polynomial MEPs from \cite{HMP2} to compute just a subset of eigenvalues
\ch{of \eqref{eq:quadprojrecttwomep}}
close to a given target \ch{that hopefully contains a subset of eigenvalues of
\eqref{eq:quadrecttwomep}).}

\subsection{\ch{Transformation to a linear RMEP}}\label{subs:lin_q2ep}
An approach that leads to smaller $\Delta$-matrices is to apply a linearization directly to
\eqref{rectmep} and transform the problem into a linear RMEP that can be solved by \ch{Algorithm~1 or}
Algorithm~2, for instance. If we
consider \eqref{eq:quadrecttwomep}, then we linearize the problem
as a linear
R2EP $(A+\lambda B_1 + \mu B_2) \, \bz = \zero$ with matrices of size
$(3n+1) \times 3n$, where $\bz = [\bx^T, \, \lambda \bx^T, \, \mu \bx^T]^T$ and, for instance,
\begin{equation} \label{eq:matABC3}
A = \smtxa{ccc}{A_{00} & A_{10} & A_{01} \\[0.5mm]
 0 & -I & 0\\
 0 & 0 & -I}, \quad
B_1 = \smtxa{ccc}{0 & A_{20} & A_{11} \\[0.5mm]
I & 0 & 0 \\
0 & 0 & 0}, \quad
B_2= \smtxa{ccc}{ 0 & 0  & A_{02} \\
 0 & 0 & 0  \\
I & 0  & 0 }.
\end{equation}
Note that the matrices in the first block row of \eqref{eq:matABC3} have size $(n+1) \times n$, while \ch{the other matrices} are of size $n \times n$.
The compression in Algorithm~2 leads
to matrices $D_0, D_1, D_2$ of size $\frac32 n(3n+1)$. The obtained
GEPs are again
singular, which means that in Step~4 of Algorithm~2 we have to apply \ch{the} staircase algorithm from
\cite{MP2}.
This algorithm returns
$2 n(n+1)$ eigenvalues, which are all solutions of \eqref{eq:quadrecttwomep}, thus there are no redundant solutions.

\subsection{Vandermonde compression} \label{subs:vand_compr}
We can exploit the Vandermonde type structure
$\bz = [\bx^T, \, \lambda \bx^T, \, \mu \bx^T]^T$ of the eigenvector in \eqref{eq:matABC3}
to compress the matrices even more. This gives \ch{us} the third approach, which is the most
efficient one.
In a similar way as in \Cref{subs:compress} we observe that solutions of \eqref{eq:quadrecttwomep} lead to an invariant subspace
spanned by all vectors of the form
$[\bx^T, \, \lambda \bx^T, \, \mu \bx^T]^T\otimes [\bx^T, \, \lambda \bx^T, \, \mu \bx^T]^T$,
where $\lambda,\mu\in \C$ and $\bx \in \C^n$. There exists a
$9n^2\times 9n^2$ permutation matrix $S$, such that
\[
\mtxb{\bx \\ \lambda \bx \\ \mu \bx}\otimes
\mtxb{\bx \\ \lambda \bx \\ \mu \bx}
= S\left( \mtxb{1 \\ \lambda \\ \mu}\otimes
\mtxb{1 \\ \lambda \\ \mu}\otimes \bx \otimes \bx \right).
\]
Let $\be_i$ denote the $i$th canonical basis vector of the appropriate dimension.
It has the form $S=I_3\otimes K\otimes I_n$,
where $K$ is a $3n \times 3n$ permutation matrix
\[ K=[\be_1 \ \ \be_4 \ \, \dots \ \, \be_{3n-2} \ \ \be_2 \ \ \be_5 \ \, \dots \ \, \be_{3n-1} \ \ \be_3 \ \ \be_6 \ \, \dots \ \, \be_{3n}],
\]
such that
$[\bx^T, \, \lambda \bx^T, \, \mu \bx^T]^T
=K(\,[1, \, \lambda, \, \mu]^T\otimes \bx \,)$.
We know from \Cref{subs:compress} that there exists an $n^2\times \frac12 n(n+1)$ matrix $T$ such that
$\bx \otimes\bx = T\bw$ for $\bw\in \C^{n(n+1)/2}$. Additionally, we take a
$9\times 6$ matrix $R$ with elements $0$ and $1$ such that
\begin{equation} \label{eq:matrika_R}
[1, \, \lambda, \, \mu]^T\otimes
[1, \, \lambda, \, \mu]^T = R \,
[1, \, \lambda, \, \mu, \, \lambda^2, \, \lambda\mu, \, \mu^2]^T.
\end{equation}
By combining $S$, $R$, and $T$, we get that\
\[
\bz\otimes \bz=S \, (R\otimes T) \ (\,[1, \, \lambda, \, \mu, \, \lambda^2, \, \lambda\mu, \, \mu^2]^T\otimes \bw\,);
\]
$\wt T:=S(R\otimes T)$ is a $9n^2\times 3n(n+1)$ matrix that we can use
to compress the size of the $\Delta$-matrices from $9n^2$ to $3n(n+1)$.
Let
$\wt\Delta_0 =B_1 \otimes B_2-B_2 \otimes B_1$,
$\wt\Delta_1=B_2 \otimes A-A\otimes B_2$, and
$\wt\Delta_2=B_1 \otimes A-A\otimes B_1$ of
size $(3n+1)^2\times 9n^2$ be the corresponding rectangular $\Delta$-matrices for \eqref{eq:matABC3}.
Similar as in \Cref{lem:rows_of_D0},
some of the rows of the matrices $\wt \Delta_i \, \wt T$
are either 0 or equal to other rows up to a sign. In particular, if we label the rows of the matrices
$A$, $B_1$, $B_2$ in
\eqref{eq:matABC3} with the elements of the
vector
\[ [y_1, \, \dots, \, y_{n+1}, \ \, s_1, \, \dots, \, s_n, \ \, t_1, \, \dots, \, t_n]^T,
\]
then the rows of $(3n+1)^2\times 3\chh{n}(n+1)$ matrices $\wt \Delta_i \, \wt T$
have labels that are of the following possible types:
$(y_j,y_k), (y_j, s_p), (y_j, t_p), (s_p, y_k), \dots, (t_p, t_q)$,
where $j,k= 1,\dots,n+1$ and $p,q= 1,\dots,n$.
The following new generalization of \Cref{lem:rows_of_D0} can be observed:
\vspace{-2mm}
\begin{enumerate} \itemsep=-1mm
\item[a)] $\wt \Delta_i \, \wt T((y_j,y_k),:) =-\wt \Delta_i \, \wt T((y_k,y_j),:)$, thus
$\wt \Delta_i \, \wt T((y_j,y_k),:) = 0$ for $j=k$;
\item[b)] $\wt \Delta_i \, \wt T((y_j, s_p),:) =-\wt \Delta_i \, \wt T((s_p,y_j):)$;
\item[c)] $\wt \Delta_i \, \wt T((y_j, t_p),:) =-\wt \Delta_i \, \wt T((t_p,y_j):)$;
\item[d)] $\wt \Delta_i \, \wt T((s_p, s_q),:) = 0$;
\item[e)] $\wt \Delta_i \, \wt T((t_p, t_q),:) = 0$;
\item[f)] $\wt \Delta_i \, \wt T((s_p, t_q),:) = \wt \Delta_i \, \wt T((s_q, t_p),:) =-\wt \Delta_i \, \wt T((t_p, s_q),:) =-\wt \Delta_i \, \wt T((t_q, s_p),:)$.
\end{enumerate}
It follows \ch{from this generalization} that we should select the rows of $\wt \Delta_i \, \wt T$ that correspond to indices labeled
as $(y_j,y_k)$, where $j<k$, $(y_j,s_p)$, $(y_j,t_q)$, and $(s_p,t_q)$ for $p<q$. This gives exactly
$3n(n+1)$ rows. Let $\wt L$ be the corresponding $3n(n+1) \times (3n+1)^2$ selection matrix with elements from $\{0,1\}$ that
generalizes the $L$ from \Cref{ex:exTL}. The $3n(n+1) \times 3n(n+1)$ matrices
$\wt D_i = \wt L \, \wt \Delta_i \, \wt T$, $i = 0,1,2$, can now be used to solve \eqref{eq:quadrecttwomep}.
The $\wt D_i$ are larger than the number of solutions of \eqref{eq:quadrecttwomep} and the GEPs $\wt D_1 \bu = \lambda \, \wt D_0 \bu$ and $\wt D_2 \bu = \mu \, \wt D_0 \bu$ are singular. Again, the staircase algorithm from \cite{MP2} can be applied to compute
$2 n(n+1)$ eigenvalues, which are all solutions of \eqref{eq:quadrecttwomep}; there are no redundant solutions.

\begin{example} \label[example]{ex:quad1}\rm We consider the quadratic R2EP
\begin{equation}
\smtxa{cc}{1 + \ph{1}\lambda + 4\mu + 2\lambda^2 + \ph{1}\lambda\mu + 3\mu^2 & \quad 2 + 3\lambda + \ph{1}\mu + 3\lambda^2 +\ph{1}\lambda\mu + \ph{1}\mu^2\\[0.5mm]
3 + 5\lambda + \ph{1}\mu + \ph{1}\lambda^2 + 2\lambda\mu + 3\mu^2 & \quad 4 + \ph{1}\lambda + 3\mu + \ph{1}\lambda^2 +2\lambda\mu + 2\mu^2\\[0.5mm]
3 + \ph{1}\lambda + 4\mu + \ph{1}\lambda^2 + \chh{2}\lambda\mu + 1\mu^2 & \quad 1 + 4\lambda + \ph{1}\mu + 2\lambda^2 +3\lambda\mu + 2\mu^2} \,
\smtxa{c}{x_1 \\ x_2} = \zero. \label{ex1:quadmatrices}
\end{equation}
We multiply the $49\times 36$ matrices $\wt \Delta_0, \wt \Delta_1, \wt \Delta_2$ on the right by a
$36\times 18$ matrix $\wt T=S(R\otimes T)$, and on the left by a $18\times 49$ matrix $\wt L$.
Here $S=I_3\otimes K\otimes I_2$, where $K=[\be_1\ \be_4\ \be_2\ \be_5\ \be_3\ \be_6]$ is a $6\times 6$ matrix,
$R$ is a $9\times 6$ matrix from \eqref{eq:matrika_R}
and $T$ is a $4\times 3$ matrix as in \Cref{ex:exTL}. Matrix $\wt L$ has one nonzero element 1 in each row, and the columns of nonzero elements
are 2--7, 10--14, 18--21, 27--28, and 35 corresponding to linearly independent rows of $\wt \Delta_i \, \wt T$ for $i = 0,1,2$.
This renders $18\times 18$ matrices $\wt D_i = \wt L \, \wt \Delta_i \, \wt T$ for $i = 0,1,2$.
The staircase algorithm from \cite{MP2}, applied to
 a pair of singular GEPs $\wt D_1 \bu = \lambda \, \wt D_0 \bu$
 and $\wt D_2 \bu = \mu \, \wt D_0 \bu$, returns
the following 12
eigenvalues of \eqref{ex1:quadmatrices}.

{\footnotesize
\begin{center}
\begin{tabular}{ll} \\ \hline \rule{0pt}{2.1ex}%
$~~~~~\lambda$ & $~~~~~\mu$ \\ \hline \rule{0pt}{2.3ex}%
$-7.5148 \pm 10.2523i$ & $-3.8435 \mp 2.4388i$ \\
$-7.6951 \pm 1.3198i$ & $\phantom{-}6.3264 \pm 2.2203i$ \\
$\phantom{-}0.3122 \pm 0.1675i$ & $-0.6460 \mp 1.2328i$ \\
$-0.1483 \pm 0.8975i$ & $-0.8786 \pm 0.1559i$ \\
$-0.8086 \pm 0.3135i$ & $-0.1788 \pm 0.6154i$ \\
$\phantom{-}0.6829$ & $\phantom{-}0.7594$ \\
$-0.9391$ & $-1.0037$ \\ \hline
\end{tabular}
\end{center}}
\end{example}

\begin{example}\rm
We compare the block Macaulay method applied to the quadratic R2EP \eqref{eq:quadrecttwomep} \ch{with random matrices to} the methods proposed in this section.
The block Macaulay method requires degree $m^*=2n+1$ and
matrices of size $(2n^3+3n^2+n) \times (2n^3+5n^2+3n)$ to solve
\eqref{eq:quadrecttwomep} using a linear polynomial for a shift.
In contrast, all methods suggested in this section
use matrices of size $\calo(n^2)$. A comparison for
$n=4,8,\dots,20$ is given in Table \ref{tab:comp_rqep}.

\begin{table}[htb!]
\centering
{\footnotesize
\caption{Sizes of the matrices of the block Macaulay method and the methods from Sections~\ref{subs:trans_q2ep}, \ref{subs:lin_q2ep}, and
\ref{subs:vand_compr} required to solve a generic quadratic
R2MEP \eqref{eq:quadrecttwomep} with matrices of size $(n+1) \times n$ for $n=2,4,\dots,20$.
One dimension is given for square matrices.} \label{tab:comp_rqep}
\begin{tabular}{r|cccc|c} \hline \rule{0pt}{2.7ex}%
$n$ & Block Macaulay & Sec. 5.2 ($\Delta_i$)& Sec. 5.3 ($D_i$)& Sec. 5.4 ($\wt D_i$) & \# Eigs
\\
\hline \rule{0pt}{2.5ex}%
$\ph{1}4$ & $\ph{11}180\times \ph{11}220$ & $\ph{1}144$ & $\ph{11}78$ & $\ph{11}60$ & $\ph{1}40$ \\
$\ph{1}8$ & $\ph{1}1224\times \ph{1}1368$ & $\ph{1}576$ & $\ph{1}300$ & $\ph{1}216$ & $144$ \\
$12$   & $\ph{1}3900\times \ph{1}4212$ & $1296$ & $\ph{1}666$ & $\ph{1}468$ & $312$ \\
$16$   & $\ph{1}8976\times \ph{1}9520$ & $2304$ & $1176$ & $\ph{1}816$ & $544$ \\
$20$   & $17220\times 18060$ & $3600$ & $1830$ & $1260$ & $840$ \\[0.5mm]
$n$   & $n(n+1)(2n+1)\times~n(n+1)(2n+3)$ & $9n^2$ & $\frac32 n(3n+1)$ & $3n(n+1)$ & $2n(n+1)$ \\[1mm] \hline
\end{tabular}}
\end{table}
\end{example}

In the next two sections we will consider particular polynomial RMEPs \eqref{rectmep} related
to finding optimal parameters for ARMA and LTI models. We will see that in some cases, where not all monomials of \ch{total} degree less or equal to $d$
are present, we can find \ch{even} more efficient linearizations resulting in smaller $\Delta$-matrices.

\section{ARMA model} \label{sec:ARMA}
Let $y_1,\dots,y_N\in \R$ be a sequence of $N$
values of a time series, which may be contaminated by noise.
There exist various models for the statistical analysis
of time series using one or more parameters;
see, e.g., \cite{BoxJenkins, Ljung}.
We want to find the optimal values of these parameters that minimize the error.
De Moor \ch{and Vermeersch} show in
\cite{DeMoor_LTI, DeMoor_CIS, DeMoor_ARMA, DeMoor_LAA} that critical points of LTI and ARMA models are eigenvalues of polynomial RMEPs,
and use the block Macaulay matrices to compute them. While
state-of-the-art numerical methods for the identification of parameters in LTI and ARMA models, based on
nonlinear optimization,
converge locally without guarantee to find the optimal solution, the solutions of
the associated \ch{polynomial} RMEPs give all stationary points including the global minimizer.
We will show that we can compute the stationary points of ARMA and LTI models (in this and the next section) more efficiently by using \ch{the} numerical methods from \Cref{subs:num_polyrect}.

The following review of the ARMA model and its relation to an RMEP is based on \cite{DeMoor_ARMA}. The scalar ARMA$(p,q)$ model is
\begin{equation} \label{eq:mod_ARMA}
\sum_{i = 0}^p\alpha_i \, y_{k-i} = \sum_{j= 0}^q \gamma_j \, e_{k-j}, \qquad k = p+1, \dots, N,
\end{equation}
where $p$ and $q$ are the orders of the autoregressive (AR) and the moving-average (MA) part, respectively.
We may assume $\alpha_0 = \gamma_0 = 1$. For a given $\by\in \R^N$, \chh{the goal is to find}
the values of the real parameters
$\alpha_1,\dots,\alpha_p$ and $\gamma_1,\dots,\gamma_q$ that minimize $\|\be\|$ of the error $\be\in \R^{N-p+q}$.
Of interest are the solutions such that the zeros of the characteristic polynomials
of AR and MA part are in the open unit disk; in the rest of this section, we will mention the constraints for the considered cases.
\ch{The optimality conditions for minimizing the error norm $\|\be\|$ are described in \cite{DeMoor_ARMA}.
They lead to a homogeneous system, where the unknown model parameters $\alpha_i$ and $\gamma_j$ appear polynomially up to degree $2$.
These polynomial equations form a quadratic RMEP whose eigenvalues are critical values for the objective function.
Note that we are interested in real eigenvalues only because the parameters of the ARMA model are real. We will now consider some instances of the problem.}

\subsection{ARMA(1,1)} \label{sec:arma11} For $p=q=1$, the corresponding
quadratic R2EP has the form {(with the simplified notation $\alpha_1 = \alpha$ and $\gamma_1 = \gamma$)}
\begin{equation} \label{eq:ARMA11}
(A_{00}+\alpha A_{10}+\gamma A_{01} +\gamma^2 A_{02}) \, \bx = \zero
\end{equation}
with matrices $A_{ij}$ of size $(3N-1) \times (3N-2)$, where \cite{DeMoor_ARMA}
\begin{equation} \label{A02}
A_{00} = \smtxa{cccc}{
\by_{(2)} & I & 0  & 0 \\[0.3mm]
\by_{(1)} & 0 & I &  0\\[0.3mm]
 0 & R & 0 & I \\[0.3mm]
 0 & \by_{(1)}^T & \by_{(2)}^T & 0  \\[0.3mm]
 0 & 0  & 0  & \by_{(2)}^T
},\quad
A_{10}= \smtxa{cccc}{
\by_{(1)} & 0 & 0 & 0 \\[0.3mm]
0 & 0 & 0 & 0 \\[0.3mm]
0 & 0 & 0 & 0 \\[0.3mm]
0 & 0 & \by_{(1)}^T & 0 \\[0.3mm]
0 & 0 & 0 & \by_{(1)}^T
}, \quad
\end{equation}
\[A_{01} = \smtxa{cccc}{
0 & R & 0 & 0 \\[0.3mm]
0 & 0 & R & 0 \\[0.3mm]
0 & 2I & 0 & R \\[0.3mm]
0 & 0 & 0 & 0 \\[0.3mm]
0 & 0 & 0 & 0
},\quad
A_{02}= \smtxa{cccc}{
0 & I & 0 & 0 \\[0.3mm]
0 & 0 & I & 0 \\[0.3mm]
0 & 0 & 0 & I\\[0.3mm]
0 & 0 & 0 & 0 \\[0.3mm]
0 & 0 & 0 & 0
},
\]
$\by_{(1)}= [y_1, \dots, y_{N-1}]^T$,
$\by_{(2)}= [y_2, \dots, y_{N}]^T$,
$I$ is the identity matrix, and
$R$ is a tridiagonal matrix with stencil $[1, \, 0, \, 1]$.
{For ARMA(1,1), of interest are \ch{real} solutions $(\alpha, \gamma)$ in the (open) square domain $(-1,1)^2$.}

Following the proof of \Cref{lem:fin_sol} we see that problem \eqref{eq:ARMA11} has $n(n+1)$ eigenvalues for generic matrices $A_{00},A_{10},A_{01},A_{02}$ of size
$(n+1) \times n$, \ch{where $n=3N-1$}. This is half the number of a generic
quadratic R2EP \eqref{eq:quadrecttwomep}; the difference
is due to the absence of the monomials $\alpha^2$ and $\alpha\gamma$
so that some solution are located at infinity.
We will show
in \Cref{ex:arma11_table} that in the particular case related to ARMA(1,1),
the number of eigenvalues is even smaller, due
to the structure of \eqref{A02}.

We may apply methods from \Cref{subs:num_polyrect} to \eqref{eq:ARMA11}, but as only one of the quadratic monomials is present, we can solve \eqref{eq:ARMA11} more efficiently.
The same problem appears in \cite[Sec.~5.4]{HMP}, with the key difference that now the matrices are rectangular.
As in \cite{HMP}, we introduce a new eigenvalue variable $\xi = \gamma^2$ and treat the
problem as a linear three-parameter eigenvalue problem of the form
\begin{align}
\left(\smtxa{ccc}{0 & 0 \\ 1 & 0} +
\alpha \, \smtxa{ccc}{0 & 0 \\ 0 & 0} +
\gamma \, \smtxa{ccc}{1 & 0 \\ 0 & 1} +
\xi \, \smtxa{ccc}{0 & 1 \\ 0 & 0}\right) \bv
& = \zero,\nonumber\\[-2.5mm]
\label{mep3:ARMA11_New} & \\[-2.5mm]
(A_{00}+ \alpha \, A_{10} + \gamma \, A_{01} + \xi \, A_{02}) \, \bx& = \zero. \nonumber
\end{align}
The first equation in \eqref{mep3:ARMA11_New} has a nonzero $\bv$ of the form $[\gamma, \, -1]^T$ when
$\xi = \gamma^2$.
\ch{Note that in \eqref{mep3:ARMA11_New} we combine rectangular and square multiparameter pencils, which is a new approach.}
Similarly to \Cref{subs:compress}, we then introduce the $2(n+1)^2\times 2n^2$ matrices
\begin{align*}
\wt \Delta_0 & =
\left[\begin{matrix}
A_{02}\otimes A_{\ch{10}} - A_{\ch{10}}\otimes A_{02} & A_{10}\otimes A_{01} - A_{01}\otimes A_{10} \\
 0 & A_{02}\otimes A_{\ch{10}} - A_{\ch{10}}\otimes A_{02}
 \end{matrix}\right],\\
\wt \Delta_1& =
 \left[\begin{matrix}
 A_{00}\otimes A_{02} - A_{02}\otimes A_{00} & A_{01}\otimes A_{00} - A_{00}\otimes A_{01} \\
 A_{02}\otimes A_{01} - A_{01}\otimes A_{02} & A_{00}\otimes A_{02} - A_{02}\otimes A_{00}
 \end{matrix}\right],\\
\wt \Delta_2& =
\left[\begin{matrix}
0 & A_{10}\otimes A_{00} - A_{00}\otimes A_{10} \\
 A_{02}\otimes A_{10} - A_{10}\otimes A_{02} & 0
 \end{matrix}\right],\\
\wt \Delta_3& =
\left[\begin{matrix}
 A_{10}\otimes A_{00} - A_{00}\otimes A_{10} & 0 \\
 A_{01}\otimes A_{00} - A_{00}\otimes A_{01} & A_{10}\otimes A_{00} - A_{00}\otimes A_{10}
 \end{matrix}\right].
\end{align*}
Solutions of \eqref{mep3:ARMA11_New} satisfy
\begin{equation} \label{eq:ARMA11_Delta_tilda}
\wt \Delta_1 \bz = \alpha \, \wt \Delta_0 \bz,\quad
\wt \Delta_2 \bz = \gamma \, \wt \Delta_0 \bz,\quad
\wt \Delta_3 \bz = \xi \, \wt \Delta_0 \bz,
\end{equation}
where $\bz = \bv\otimes \bx \otimes \bx$. Let $T$ be the right compression matrix of size
 $n^2\times \frac12 n(n+1)$ from Algorithm~2 such that $\bx \otimes\bx = T\bw$ for $\bw\in \C^{n(n+1)/2}$ and let $L$ be the corresponding left compression matrix
of size $\frac12 n(n+1) \times (n+1)^2$. We apply $T$ and $L$ to compress \eqref{eq:ARMA11_Delta_tilda} into a
system of singular GEPs
\begin{equation} \label{eq:ARMA11_D_tilda}
\wt D_1 \bu = \alpha \, \wt D_0 \bu,\quad
\wt D_2 \bu = \gamma \, \wt D_0 \bu,\quad
\wt D_3 \bu = \xi \, \wt D_0 \bu,
\end{equation}
where $\bu= \bv\otimes T\bw$ and
$\wt D_i = (I_2 \otimes L) \, \wt\Delta_i \, (I_2 \otimes T)$
for $i = 0,1,2,3$ are matrices of size $n(n+1) \times n(n+1)$.
We solve \eqref{eq:ARMA11_D_tilda} by the staircase algorithm from \cite{MP2}, which works on the set of all these generalized eigenproblems.

\begin{example}\rm\label[example]{ex:arma11}
We take $\by\in \R^{12}$, where
\begin{center}
$\by =$ {\small \begin{tabular}{llllll}
[\ 2.4130 & 1.0033 & 1.2378 & $-0.72191$ & $-0.81745$ & $-2.2918$ \\
\ph{[\ }0.18213 & 0.073557 & 0.55248 & $\ph{-}2.0180$ & $\ph{-}2.6593$ & $\ph{-}1.1791\ ]^T$
\end{tabular}}
\end{center}
and construct matrices $A_{00}, A_{10}, A_{01}, A_{02}$ of size $35\times 34$ according to \eqref{A02}. Using $T$ of size $1156\times 595$ and $L$ of size $595\times 1225$ we construct matrices $\wt D_0,\wt D_1,\wt D_2,\wt D_3$ of size $1190\times 1190$ from
\eqref{eq:ARMA11_D_tilda}. We solve \eqref{eq:ARMA11_D_tilda} by the staircase algorithm from \cite{MP2}. In this particular example, most time is spent \chh{in the first two steps of the staircase algorithm} on two singular value decompositions
(alternatively, we could use a slightly cheaper rank-revealing QR) \chh{of} matrices of sizes $1190\times 1190$ and $1190\times 2988$ respectively. We find 147 eigenvalues of \eqref{eq:ARMA11_D_tilda}, which are all eigenvalues of the RMEP \eqref{eq:ARMA11}.
Only three eigenvalues are real. They are all inside the domain of interest and they give the following local minimum and two saddle points of the objective function \ch{$\|\be\|^2$}.

{\footnotesize
\begin{center}
\begin{tabular}{l|ccc} \hline \rule{0pt}{2.3ex}%
{Type stationary point} & $\alpha$ & $\gamma$ & \ch{$\|\be\|^2$} \\[0.5mm] \hline \rule{0pt}{2.3ex}%
{Saddle point} & $\ph{-}0.3224$ & $\ph{-}0.7799$ & 17.58 \\
Local minimum & $-0.5234$ & $\ph{-}0.0476$ & 13.85 \\
{Saddle point} & $-0.8305$ & $-0.8542$ & 23.78 \\ \hline
\end{tabular}
\end{center}
}

\medskip\noindent
Figure \ref{fig:arma11} confirms that these are indeed the critical points for the ARMA(1,1) model and given data $\by$.
The obtained result is more accurate than the solution obtained by the {\tt armax} function in MATLAB's System Identification Toolbox, which
returns $\alpha=-0.6868$, $\gamma= 0.01091$, and $\ch{\|\be\|^2}= 14.39$. This agrees with the observation in
\cite{DeMoor_ARMA} that values obtained via RMEPs are \chh{closer to the global minimum than that obtained by} {\tt armax}, which uses a nonlinear optimization algorithm described in \cite[Ch.~7]{Ljung}.

\begin{figure}
\begin{center}
\includegraphics[width=80mm]{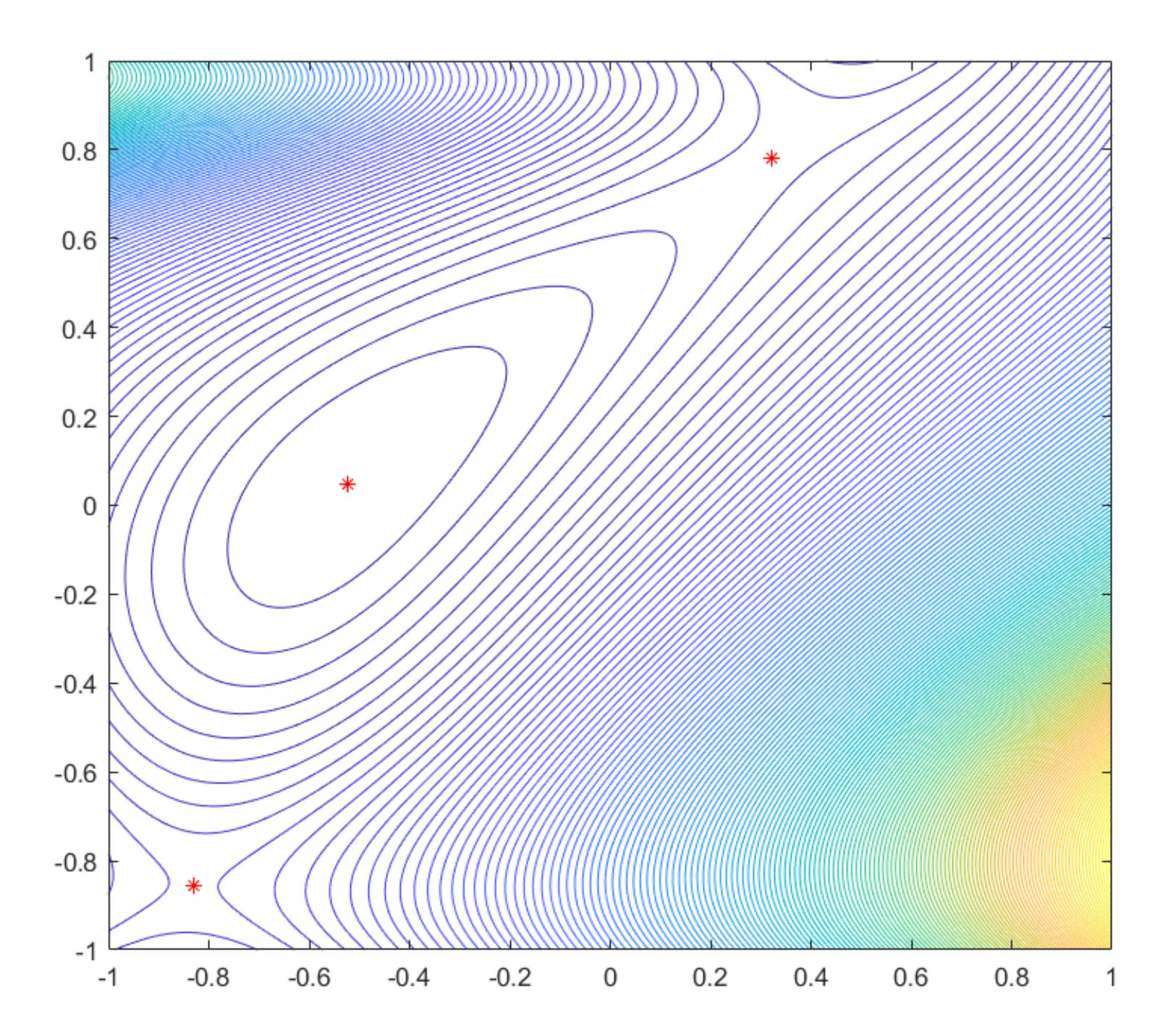}
\vspace{-5mm}
\end{center}
\caption{Contour plot of the objective function \ch{$\|\be\|^2$} for $(\alpha,\gamma) \in [-1,1]^2$ in \Cref{ex:arma11} with computed stationary points, \ch{indicated by a red star}.} \label{fig:arma11}
\vspace{-4mm}
\end{figure}

\end{example}

\begin{example}\rm\label[example]{ex:arma11_table}
We have repeated the above example for several random vectors $\by\in \R^N$ for different $N$;
for all samples the above approach is able to compute the real critical points. As we get the same number of eigenvalues of \eqref{eq:ARMA11}
for all vectors $\by$ of the same length, and the numbers seem to follow a linear pattern---if we increase $N$ by one,
the number of eigenvalues increases by 14---we believe that these are correct numbers of eigenvalues for a generic ARMA(1,1) case.
The results are presented in \Cref{tab:arma11}. For each $N$, we generate 10 random vectors $\by$ in Matlab as {\tt y=randn(N,1)}. In the second column, we give the size
of the $\wt D_i$ from \eqref{eq:ARMA11_D_tilda}, in the third column the number of detected eigenvalues, and in the fourth column
the average time required to solve the ARMA(1,1) problem of a given size.
\chh{All computations in this paper have} been performed in MATLAB 2021b on a PC with 32 GB RAM and i7--11700K 3.6 GHz CPU.

We have also applied the block Macaulay matrix approach from \cite{DeMoor_ARMA}. Both methods return the same critical points, but as the block Macaulay matrices
are much larger than the $\wt D_i$-matrices, we are able to use this method only for small $N$. In \Cref{tab:arma11}, we give the sizes of the matrices and numbers of computed eigenvalues for generic samples $\by$.
The sizes of the block Macaulay matrices for $N\ge 8$ are estimated on a hypothesis established through observations that the required degree is $\ch{m^*=}6(N-1)+1$.
While the sizes of the \chh{$\wt D_i$}-matrices grow with $\calo(N^2)$, the block Macaulay matrices grow as $\calo(N^3)$. We do not report the times for the much slower block Macaulay approach.
Compared to the results in \cite{DeMoor_2023}, where a computation using
a block Macaulay matrix for $N=8$ 
took 41.7 seconds
on a 3.2 GHz M1 CPU MacBook Pro, the new approach is clearly more efficient.

\begin{table}[htb!]
\centering
{\footnotesize \caption{Matrix sizes required to find stationary points for the generic ARMA(1,1) model by the method from \Cref{sec:arma11} and by the block Macaulay approach for $N=4,6,\dots,20$.
Starred elements are estimates.
} \label[table]{tab:arma11}
\vspace{1mm}
\begin{tabular}{cc|cc|cc} \hline \rule{0pt}{2.6ex}%
$N$ & \# Eigs & $\wt D$ & Time (s)
& Degree & Macaulay matrix \\ \hline \rule{0pt}{2.3ex}%
\ph{1}4 & $\ph{1}35$ & $\ph{1}110$ & $0.008$ & \ph{1}19\ph{*} & $\ph{11}1881\times \ph{11}2100$ \\
\ph{1}6 & $\ph{1}63$ & $\ph{1}272$ & $0.027$ & \ph{1}31\ph{*} & $\ph{11}7905\times \ph{11}8448$ \\
\ph{1}8 & $\ph{1}91$ & $\ph{1}506$ & $0.085$ & \ph{1}43* & $\ph{1}20769\times \ph{1}21780$ \\
10 & $119$ & $\ph{1}812$ & $0.225$ & \ph{1}55* & $\ph{1}43065\times \ph{1}44688$ \\
12 & $147$ & $1190$ & $0.595$ & \ph{1}67* & $\ph{1}77385\times \ph{1}79764$ \\
14 & $175$ & $1640$ & $1.43\ph{0}$ & \ph{1}79* & $126321\times 129600$ \\
16 & $203$ & $2162$ & $3.51\ph{0}$ & \ph{1}91* & $192465\times 196788$ \\
18 & $231$ & $2756$ & $7.41\ph{0}$ & 103* & $278409\times 283920$ \\
20 & $259$ & $3422$ & $14.1\ph{000}$ & 115* & $386745\times 393588$ \\ \hline
\end{tabular}}
\end{table}

We note that a generic ARMA(1,1) problem, due to the special structure \eqref{A02} of the matrices, has far fewer finite solutions than a
generic quadratic R2EP with monomials $1,\lambda,\mu,\mu^2$ and matrices
of the same size, since some of the solutions are at infinity.
\end{example}

\subsection{ARMA(2,1)} The choice $p=2$ and $q= 1$ leads to
a quadratic three-parameter RMEP (with the simplified notation $\gamma_1 = \gamma$)
\begin{equation} \label{eq:arma21}
(A_{000}+\alpha_1 A_{100} + \alpha_2 A_{010} + \gamma A_{001} + \gamma^2 A_{002}) \, \bx = \zero,
\end{equation}
with matrices of size $(4N-5) \times (4N-7)$ that we get from the optimality conditions
in an analogous way as for ARMA(1,1) \ch{(for the exact construction of the matrices, see \cite{DeMoor_ARMA})}.
For ARMA(2,1), the solutions of interest should satisfy $|\gamma| < 1$, and the roots of $t^2 + \alpha_1 t + \alpha_2$ should be in the open unit disk.
We introduce $\xi = \gamma^2$ similar to \eqref{mep3:ARMA11_New} and write
the problem as a linear four-parameter eigenvalue problem
\begin{align}
 \left(\smtxa{ccc}{0 & 0 \\ 1 & 0} +
 \alpha_1 \, \smtxa{ccc}{0 & 0 \\ 0 & 0} +
 \alpha_2 \, \smtxa{ccc}{0 & 0 \\ 0 & 0} +
 \gamma \, \smtxa{ccc}{1 & 0 \\ 0 & 1} +
 \xi \, \smtxa{ccc}{0 & 1 \\ 0 & 0}\right) \bv
 & = \zero,\nonumber\\[-2.5mm]
 \label{mep3:ARMA21_New} & \\[-2.5mm]
 (A_{000}+ \alpha_1 \, A_{100}+ \alpha_2 \, A_{010} + \gamma \, A_{001} + \xi \, A_{002}) \, \bx& = \zero,\nonumber
\end{align}
\ch{where we again combine the equations with rectangular and square matrices.}
If we form the corresponding $\wt \Delta_i$-matrices then
solutions of \eqref{mep3:ARMA21_New} satisfy
\begin{equation} \label{eq:ARMA21_Delta_tilda}
\wt \Delta_1 \bz = \alpha_1 \, \wt \Delta_0 \bz,\quad
\wt \Delta_2 \bz = \alpha_2 \, \wt \Delta_0 \bz,\quad
\wt \Delta_3 \bz = \gamma \, \wt \Delta_0 \bz,\quad
\wt \Delta_4 \bz = \xi \, \wt \Delta_0 \bz,
\end{equation}
where $\bz = \bv\otimes \bx \otimes \bx \otimes \bx$. Using the right compression matrix $T$ of size
$n^3\times \frac16 n(n+1)(n+2)$, where $n=4N-7$, from Algorithm~2 such that $\bx \otimes\bx \otimes\bx = T\bw$ for $\bw\in \C^{n(n+1)(n+2)/6}$ and the corresponding left compression matrix $L$
of size $\frac16 n(n+1)(n+2) \times (n+2)^3$, we compress \eqref{eq:ARMA21_Delta_tilda} into a
system of singular GEPs
\begin{equation} \label{eq:ARMA21_D_tilda}
\wt D_1 \bu = \alpha_1 \, \wt D_0 \bu,\quad
\wt D_2 \bu = \alpha_2 \, \wt D_0 \bu,\quad
\wt D_3 \bu = \gamma \, \wt D_0 \bu,\quad
\wt D_4 \bu = \xi \, \wt D_0 \bu,
\end{equation}
where $\bu= \bv\otimes T\bw$ and
$\wt D_i = (I_2 \otimes L) \, \wt\Delta_i \, (I_2 \otimes T)$.

Since the $\wt D_i$-matrices are of size $\frac16(4N-7)(4N-6)(4N-5)$, we
are able to use this approach only for small $N$. For instance, for
$N=5, \dots, 10$ we get the $\wt D_i$-matrices of size $910, 1938, 3542, 5850, 8990, 13090$, respectively.

\begin{example}\rm
We take
\[\by = \fns
 [0.41702,\ 0.72032,\ 0.01234,\ 0.30233,\ 0.14676,\  0.09234,\  0.18626]^T
 \]
and construct $A_{000}, A_{100}, A_{010}, A_{001}, A_{002}$ of size $23\times 21$ for \eqref{eq:arma21}. Using matrices
$T$ of size $9261\times 1771$ and $L$ of size $1771\times 12167$ we construct $\wt D_0,\wt D_1,\wt D_2,\wt D_3,\wt D_4$ of size $3542\times 3542$.
A staircase type algorithm from \cite{MP} finds 29 eigenvalues of 
\eqref{eq:arma21} in $5.4$ seconds.
Three eigenvalues are real, they are all inside the domain of interest
and give the following critical points $(\alpha_1,\alpha_2,\gamma)$ of the objective function \ch{$\|\be\|^2$}:

{\footnotesize
\begin{center}
\begin{tabular}{l|cccc} \hline \rule{0pt}{2.2ex}%
{Type stationary point} & $\alpha_1$ & $\alpha_2$ & $\gamma$ & \ch{$\|\be\|^2$} \\[0.5mm] \hline \rule{0pt}{2.3ex}%
{Saddle point} & $0.08749$ & $-0.4354$ & $-0.2483$ & $0.055219$ \\
{Local minimum} & $0.08843$ & $-0.4188$ & $-0.1575$ & $0.055204$ \\
{Saddle point} & $0.06827$ & $-0.2921$ & $\ph{-}0.4172$ & $0.057668$ \\ \hline
\end{tabular}
\end{center}
}
\end{example}

Although \ch{this} new approach yields considerably smaller matrices compared to a block Macaulay approach, the sizes of the associated matrices still grow rapidly for larger values of $p$, $q$, and $N$.

\section{LTI} \label{sec:LTI}
In the least squares optimal realization problem of autonomous linear time-invariant systems (LTI($p$)),
we want to find
parameters $\alpha_1,\dots,\alpha_p$ that admit the best 2-norm approximation of a given $\by\in \R^N$ by $\wh \by\in \R^N$ whose elements satisfy the difference equation
\[
\wh y_{k+p}+\alpha_1 \, \wh y_{k+p-1}+\cdots+\alpha_p \,\wh y_k= 0, \qquad k= 1,\dots,N-p,
\]
where $p$ is the order of the LTI.

The optimality conditions for minimizing the error norm \ch{$\|\by-\wh\by\|$} are described in \cite{DeMoor_LTI}.
In a similar way as for the \chh{ARMA model}, the critical values $\alpha_1,\dots,\alpha_p$ of the objective function
are eigenvalues of a quadratic RMEP. As before we are interested in real eigenvalues only. We will now consider instances of the probem
for $p=2$ and $p=3$.

\subsection{LTI(2)}
An LTI(1) model leads to a standard one-parameter quadratic eigenvalue problem that can be solved with several well-known numerical methods.
When we add one parameter and consider the LTI(2) model, then
the corresponding quadratic R2EP has the form\cite{DeMoor_LTI}
\begin{equation} \label{eq:LTI2}
(A_{00}+\alpha_1 A_{10}+\alpha_2 A_{01} +\alpha_1^2 A_{20} + \alpha_1\alpha_2 A_{11} +
\alpha_2^2 A_{02}) \, \chh{\bx} = \zero,
\end{equation}
where $A_{00}$, $A_{10}$, and $A_{01}$ are the $(3N-4) \times(3N-5)$ matrices
\[
\smtxa{cccc}{
\by_{(3)} & I & 0& 0\\[0.3mm]
\by_{(2)} & R & I & 0\\[0.3mm]
\by_{(1)} & S & 0 & I \\[0.3mm]
0 & \by_{(2)}^T & \by_{(3)}^T & 0\\[0.7mm]
0 & \by_{(1)}^T& 0 & \by_{(3)}^T
}, \quad
\smtxa{cccc}{
\by_{(2)} & R & 0& 0\\[0.3mm]
 0 & 2I & R & 0\\[0.3mm]
 0 & R& 0 & R\\[0.3mm]
 0 & 0 & \by_{(2)}^T & 0\\[0.3mm]
 0 & 0 & 0 & \by_{(2)}^T
}, \quad
\smtxa{cccc}{
\by_{(1)} & S & 0 & 0\\[0.3mm]
 0 & R & S & 0\\[0.3mm]
 0 & 2I & 0 & S \\[0.3mm]
 0 & 0 & \by_{(1)}^T & 0\\[0.3mm]
 0 & 0 & 0 & \by_{(1)}^T
},
\]
respectively, $A_{20}=A_{02}$ are as in \eqref{A02}, $A_{11} = A_{02} \cdot \text{diag}(R,R,R,R)$,
$\by_{(1)}= [y_1, \dots, y_{N-2}]^T$,
$\by_{(2)}= [y_2, \dots, y_{N-1}]^T$,
$\by_{(3)}= [y_3, \dots, y_{N}]^T$,
$R$ is tridiagonal with stencil $[1, \, 0, \, 1]$,
\ch{and} $S$ is pentadiagonal with stencil $[1, \, 0, \, 0, \, 0, \, 1]$.
Based on the five-parameter linearization for the quadratic MEP in \cite[Sec.~4]{HMP}
we use a similar approach as in \Cref{sec:ARMA} and linearize \eqref{eq:LTI2} as a four-parameter eigenvalue problem by introducing
two additional parameters $\xi_1= \alpha_1\alpha_2$ and
$\xi_2= \alpha_1^2+\alpha_2^2$.
This problem has the form
\begin{align}
\left(\smtxa{ccc}{0 & 0 \\ 1 & 0} +
\alpha_1 \, \smtxa{ccc}{1 & 0 \\ 0 & 0} +
\alpha_2 \, \smtxa{ccc}{0 & 0 \\ 0 & 1}
+ \xi_1 \, \smtxa{ccc}{0 & 1 \\ 0 & 0}
+ \xi_2 \, \smtxa{ccc}{0 & 0 \\ 0 & 0}
\right) \bv_1
& = \zero, \nonumber\\
\left(\smtxa{ccc}{0 & 0 \\ 1 & 0} +
\alpha_1 \, \smtxa{ccc}{1 & 0 \\ 0 & 1} +
\alpha_2 \, \smtxa{ccc}{1 & 0 \\ 0 & 1}
+ \xi_1 \, \smtxa{ccc}{0 & 2 \\ 0 & 0}
+ \xi_2 \, \smtxa{ccc}{0 & 1 \\ 0 & 0}
\right) \bv_2
& = \zero,\label{mep4:LTI2}\\
(A_{00}+ \alpha_1 \, A_{10} + \alpha_2 \, A_{01} +
\xi_1 \, A_{11} + \xi_2 \, A_{02}) \, \bx& = \zero. \nonumber
\end{align}
The first two equations in \eqref{mep4:LTI2} have nonzero solutions $\bv_1$ and $\bv_2$, of the form $[\alpha_2, \, -1]^T$ when $\xi_1= \alpha_1\alpha_2$ and $[\alpha_1+\alpha_2, \, -1]^T$
when $\xi_2= \alpha_1^2+\alpha_2^2$, respectively.

If we form the corresponding $\wt \Delta_i$-matrices of size $4(3N-4)^2\times 4(3N-5)^2$, then
solutions of \eqref{eq:LTI2} satisfy
\begin{equation} \label{eq:LTI2_Delta_tilda}
\wt \Delta_1 \bz = \alpha_1 \, \wt \Delta_0 \bz,\quad
\wt \Delta_2 \bz = \alpha_2 \, \wt \Delta_0 \bz,\quad
\wt \Delta_3 \bz = \xi_1 \, \wt \Delta_0 \bz,\quad
\wt \Delta_4 \bz = \xi_2 \, \wt \Delta_0 \bz,
\end{equation}
where $\bz = \bv_1 \otimes \bv_2 \otimes \bx \otimes \bx$. Using the right compression matrix $T$ of size
$n^2\times \frac{1}{2}n(n+1)$, where $n=3N-5$, such that $\bx \otimes\bx = T\bw$ for $\bw\in \C^{n(n+1)/2}$ and the corresponding left compression matrix $L$
of size $\frac{1}{2}n(n+1) \times (n+1)^2$, we compress \eqref{eq:LTI2_Delta_tilda} into a
system of singular GEPs
\begin{equation} \label{eq:LTI2_D_tilda}
\wt D_1 \bu = \alpha_1 \, \wt D_0 \bu,\quad
\wt D_2 \bu = \alpha_2 \, \wt D_0 \bu,\quad
\wt D_3 \bu = \xi_1 \, \wt D_0 \bu,\quad
\wt D_4 \bu = \xi_2 \, \wt D_0 \bu,
\end{equation}
where $\bu=T\bw\otimes \bv_1 \otimes \bv_2$ and
$\wt D_i = (I_4\otimes L) \, \wt\Delta_i \, (I_4\otimes T)$.
The matrices $\wt D_i$ of size $2(3N-5)(3N-4) \times 2(3N-5)(3N-4)$ are much smaller than the matrices required to solve \eqref{eq:LTI2} using the block Macaulay method.

\begin{example} \label[example]{ex:LTI2} \rm
We take $\by\in \R^{10}$, where
\begin{center}
$\by =$ {\small \begin{tabular}{lllll}
[\ 0.69582 & $\ph{-}0.68195$ & $-0.24647$ & $\ph{-}0.50437$ & $-0.23207$ \\
\ph{[\ }0.34559 & $-0.19628$ & $\ph-0.20553$ & $-0.17737$ & $\ph{-}0.11543\ ]^T$.
\end{tabular}}
\end{center}
This vector $\by$ has been constructed as follows. First, we select a 
\ch{nonzero}
vector $\bz$ 
such that the elements of $\bz$ satisfy the difference equation $z_k+\alpha_1 z_{k-1}+ \alpha_2 z_{k-2}= 0$
\ch{for $\alpha_1 = 0.6$ and $\alpha_2 = -0.25$}.
Then we perturb $\by = \bz+\be$ with a command of the form {\tt rng(1), e=0.1*randn(10,1)}.
By this construction, $(\alpha_1, \alpha_2)$ is a good approximation to the (global) minimizer, and $\|\be\|^2= 0.05392$ is a good approximation to the minimum.

We build matrices $A_{ij}$ for $0\le i+j\le 2$ of size $26\times 25$.
The above procedure leads  to a system \eqref{eq:LTI2_D_tilda} with matrices of size $1300\times 1300$.
A staircase algorithm, applied to \eqref{eq:LTI2_D_tilda},
returns 1059 eigenvalues of \eqref{mep4:LTI2}. There are 11 real eigenvalues $(\alpha_1,\alpha_2)$ that give stationary points of the objective function \ch{$\|\by-\wh\by\|^2$}, the minimum is obtained at $(\alpha_1,\alpha_2) =(0.60076, -0.26572)$, where $\ch{\|\by-\wh\by\|^2}= 0.03991$.
As expected, this is close to the parameters of the initial vector before the perturbation.
\end{example}

\begin{example}\rm\label[example]{ex:lti2_11}
We repeat the computation of the optimal parameters for the LTI(2) model
for random vectors $\by$ of different sizes. As we obtain the same number of eigenvalues of \eqref{eq:LTI2} for all vectors $\by$ of the same size, and the numbers follow a quadratic polynomial pattern, we believe that these numbers of eigenvalues are true for a generic case.
For vectors $\by$ of small length we applied also the block Macaulay approach \ch{and obtained the same critical points.}

The results are presented in \Cref{tab:LTI2}. For each $N$ we generated 10 random vectors $\by$ in Matlab with entries from a standard normal distribution.
In the second column we give the size
of the $\wt D_i$ from \eqref{eq:LTI2_D_tilda}, in the third column the number of eigenvalues found and in the fourth column the average time required to solve the LTI(2) problem of a given size. In fourth and fifth column we give the required degree and the
size of the block Macaulay matrices.
The required degrees and sizes for $N\ge 10$ are based on an observation
\chh{that the required degree} to extract the eigenvalues is $\chh{m^*=6\,(N-2)}$.
As for the ARMA(1,1) model, the sizes of $\wt D_i$-matrices grow with $\calo(N^2)$ and the sizes of
block Macaulay matrices grow as $\calo(N^3)$.
The new method is considerably faster than the block Macaulay approach used in \cite{DeMoor_2023},
\chh{where it is reported that} a computation needed 2.3 seconds to solve the problem for $N=6$. \chh{For a comparison of
hardware used in \cite{DeMoor_2023} and in \Cref{tab:LTI2}, see \Cref{ex:arma11_table}.}

\begin{table}[htb!]
\centering
{\footnotesize \caption{Sizes of the matrices for the new approach versus the block Macaulay matrices for the LTI(2) problem in \Cref{ex:lti2_11} for $\by \in \R^N$. Starred elements are estimates.} \label{tab:LTI2}
\begin{tabular}{rc|cc|cc} \hline \rule{0pt}{2.5ex}%
$N$ & \# Eigs & $\wt D$ & Time (s)
& Degree & Macaulay matrix \\
\hline \rule{0pt}{2.3ex}%
$\ph{1}4$ & $\ph{11}51$ & $\ph{1}112$ & $0.015$ & 12\ph{*} & $\ph{11}528\times \ph{11}637$ \\
$\ph{1}6$ & $\ph{1}243$ & $\ph{1}364$ & $0.158$ & 24\ph{*} & $\ph{1}3864\times \ph{1}4225$ \\
$\ph{1}8$ & $\ph{1}579$ & $\ph{1}760$ & $1.18\ph{0}$ & 36\ph{*} & $12600\times 13357$ \\
$10$ & $1059$ & $1300$ & $7.03\ph{0}$ & 48* & $29328\times 30625$ \\
$12$ & $1683$ & $1984$  & $32.5\ph{000}$ & 60* & $56640\times 58621$ \\ \hline
\end{tabular}}
\end{table}

Compared to the ARMA(1,1) model, we see that for LTI(2) we get many more eigenvalues for the same $N$. The number of eigenvalues of \eqref{eq:LTI2}
grows quadratically with $N$, while in ARMA(1,1) the growth is linear. If we divide the number of eigenvalues
of \eqref{eq:LTI2} by the size of the $\wt D_i$-matrices, this quotient increases to one for $N \to \infty$. In contrast, for the ARMA(1,1) problem the number of eigenvalues of \eqref{eq:ARMA11} divided by
the size of the $\wt D_i$-matrices decreases to $0$.
\end{example}

\subsection{LTI(3)} Omitting the details, we \ch{also} mention that a LTI(3) model leads to a quadratic three-parameter RMEP
with $(4N-9) \times(4N-11)$ matrices of the form \ch{(see \cite{DeMoor_LTI} for the exact construction of the matrices)}
\[
(A_{000} +\alpha_1 A_{100}+\alpha_2 A_{010} +\alpha_3 A_{001} + (\alpha_1\alpha_2 + \alpha_2\alpha_3)\,A_{110} +
\alpha_1\alpha_3A_{101}+(\alpha_1^2+ \alpha_2^2+\alpha_3^2)\, A_{200}) \, \bz = \zero.
\]
By introducing the variables
$\xi_1= \alpha_1\alpha_2+\alpha_2\alpha_3$, $\xi_2= \alpha_1\alpha_3$, and
$\xi_3= \alpha_1^2+\alpha_2^2+\alpha_3^2$, we can write this as a linear six-parameter RMEP,
construct rectangular operator determinants $\wt \Delta_i$ and compress them into $\wt D_i$ similar to \eqref{eq:LTI2_Delta_tilda}--\eqref{eq:LTI2_D_tilda}.
The corresponding matrices $\wt D_0,\dots,\wt D_6$ are of size $4\binom{4N-11}{3}\times 4\binom{4N-11}{3}$, which gives
$7752\times 7752$ for the smallest nontrivial $N=7$.
In view of this size, we do not provide numerical results.

\section{Conclusions} \label{sec:concl}
Driven by the connection of ARMA and LTI models to polynomial RMEPs,
we have studied a \ch{novel} solution approach of these problems.
In a new theoretical contribution, we have counted the number of eigenvalues of a generic polynomial RMEP
in \Cref{lem:fin_sol}.

We \ch{have suggested} two approaches to numerically solve a (polynomial) RMEP. The first approach is to transform the problem into a square (polynomial) MEP
by random projections or selections of rows, and then solve it by existing numerical methods for MEPs.
The second approach is to \chh{apply} a compression  \chh{to the related operator determinants},
which leads to a system of GEPs that again can be solved by existing numerical methods. \chh{For polynomial RMEPs we introduced Vandermonde compression and demonstrate it on a quadratic R2EP.}
The second approach produces smaller matrices and is more appropriate for problems that are small
enough \ch{so} that the obtained GEPs can be solved with direct eigenvalue solvers; for a comparison, see \Cref{tab:compare}.

We have applied the second approach to compute the stationary points of ARMA and LTI models, which can be formulated as polynomial RMEPs \ch{as proposed in \cite{DeMoor_ARMA,DeMoor_LTI}}.
With this method we can compute \emph{all} stationary points, in contrast to, e.g., gradient methods.
Compared to the methods of \cite{DeMoor_LTI, DeMoor_CIS, DeMoor_ARMA, DeMoor_LAA}, which use block Macaulay matrices \ch{to solve the polynomial RMEPs}, the presented technique employs much smaller matrices; see \Cref{subs:Macaulay} for more details.
The main ideas from \Cref{sec:ARMA} and \Cref{sec:LTI} may also be applied to find critical points of more general linear models, for instance,
given data $\by\in\R^N$ and matrices $A_0,A_1\in\R^{(N-1)\times N}$ and $C_0,C_1\in\R^{(N-1)\times N}$, find a (local) minimizer $(\alpha,\gamma)\in\R^2$ for $\|\be\|$ where $(A_0+\alpha A_1)\,\by = (C_0+\gamma C_1)\,\be$.

As a byproduct of our study of RMEPs, we have introduced new linearizations of quadratic MEPs.
For instance, linearization \eqref{mep4:LTI2} is in the line of those in \cite{HMP}, but is a new case since the matrices $A_{20}$ and $A_{02}$ are identical. \ch{We also applied a new technique of considering systems of polynomial multivariate matrix equations, where in some
equations matrices are square and rectangular in others.}

Table~\ref{tab:nr} displays  a list of cases that we have discussed \ch{in this paper}.
The number of solutions in the fourth column corresponds to a generic problem of a given type.
Due to the structure of the matrices, the number of solutions for the particular applications ARMA and LTI is much smaller
\chh{than} the number given in
the fourth column as many solutions are infinite. Also, in these examples only real solutions are relevant and there are just
few such solutions for each of the problems.
The fifth column shows the sizes of the associated \ch{GEPs} after \chh{the} compression. \chh{The} quadratic R2EP in the \chh{fourth} row \chh{is the only example where we use Vandermonde compression} and also the only example with \ch{redundant} eigenvalues.

\begin{table}[htb!]
\centering
\caption{Number of eigenpairs for a few relevant RMEPs, \ch{tackled in this paper}. The problems have size $(n+k-1) \times n$, where $k$ is the number of parameters ($\lambda, \mu, \dots$).} \label{tab:nr}
\vspace{1mm}
{\footnotesize \begin{tabular}{cllcc} \hline \rule{0pt}{2.3ex}%
$k$ & Rectangular problem & Application & \# Solutions & Size of GEP \\ \hline \rule{0pt}{2.6ex}%
2 & $A+\lambda B+\mu C$ & Rank drop & $\frac12 n(n+1)$ & $\frac12 n(n+1)$ \\[0.5mm]
& $A+\lambda B+\mu C+\mu^2 D$ & ARMA(1,1) & $ n(n+1)$ & $n(n+1)$ \\[0.5mm]
& $A+\lambda B+\mu C+\lambda^2 D+\lambda \mu E+\mu^2 D$ & LTI(2) & $ 2n(n+1)$ & $\chh{2}n(n+1)$ \\[0.5mm]
& \chh{$A+\lambda B+\mu C+\lambda^2 D+\lambda \mu E+\mu^2 F$} & \chh{Rank drop} & \chh{$ 2n(n+1)$} & \chh{$3n(n+1)$} \\[0.5mm]
3 & $A+\lambda B+\mu C+\nu D$ & Rank drop & $\frac16 n(n+1)(n+2)$ & $\frac16 n(n+1)(n+2)$ \\[1mm]
& $A+\lambda B+\mu C+\nu D+\nu^2 E$ & ARMA(2,1) & $\frac13 n(n+1)(n+2)$ & $\frac13 n(n+1)(n+2)$ \\[1mm]
\ch{$k$} & \ch{$A+\lambda_1 B_1+\cdots+\lambda_k B_k$} & \ch{Rank drop} & \ch{$\binom{n+k-1}{k}$} & \ch{$\binom{n+k-1}{k}$} \\[1.5mm] \hline
\end{tabular}}
\end{table}

Numerical methods and examples for a generic linear RMEP, a generic quadratic R2EP, and for ARMA(1,1), ARMA(2,1), and LTI(2) models, based on a transformation to a MEP, are available in \ch{the Matlab toolbox} \chh{{\tt MultiParEig}} \cite{MultiParEig}, together with numerical methods for
MEPs and systems of GEPs.
Due to the size of the involved matrices, it still seems very challenging to solve LTI and ARMA problems with RMEPs for more than a few parameters or for larger values of $N$. \ch{On the other hand, these methods guarantee to find the globally optimal solution, which might be relevant for small problems.}

\section*{Acknowledgements}
T.~Ko\v{s}ir and B.~Plestenjak have been supported by the \ch{Slovenian Research and Innovation Agency}, Research Grant N1-0154.
\ch{We are grateful to the referees for suggestions that improved the document considerably, especially to one referee who has been of tremendous help to us.}

\bibliographystyle{plain}

\end{document}